\documentclass{amsart}
\usepackage[a4paper,outer=3.5cm,inner=3.5cm,top=3cm, bottom=3cm]{geometry} 
\usepackage[english]{babel}
\usepackage[utf8]{inputenc}
\usepackage{amsmath}
\usepackage{amssymb}
\usepackage{amsthm}
\usepackage{cleveref}
\usepackage{enumerate}
\usepackage[shortlabels]{enumitem}
\usepackage{mathtools}
\usepackage{tikz-cd}
\usepackage{stmaryrd}
\usepackage{verbatim}

\newcommand{\overbar}[1]{\mkern 1.5mu\overline{\mkern-1.5mu#1\mkern-1.5mu}\mkern 1.5mu}

\newtheorem{thm}{Theorem}[section]
\newtheorem{Theorem}[thm]{Theorem}
\newtheorem{Lemma}[thm]{Lemma}
\newtheorem{Proposition}[thm]{Proposition}
\newtheorem{Corollary}[thm]{Corollary}
\newtheorem{claim}[thm]{Claim}

\theoremstyle{definition}
\newtheorem{Definition}[thm]{Definition}
\newtheorem{Question}[thm]{Question}
\newtheorem{Example}[thm]{Example}
\newtheorem{Remark}[thm]{Remark}

\newcommand{\A}{\mathbb{A}}
\newcommand{\B}{\mathbb{B}}
\newcommand{\C}{\mathbb{C}}
\newcommand{\F}{\mathbb{F}}
\newcommand{\G}{\mathbb{G}}
\newcommand{\N}{\mathbb{N}}
\renewcommand{\O}{\mathcal{O}}
\renewcommand{\P}{\mathbb{P}}
\newcommand{\Q}{\mathbb{Q}}

\newcommand{\Z}{\mathbb{Z}}
\newcommand{\m}{\mathfrak{m}}

\newcommand{\Zp}{{\Z_p}}

\newcommand{\RGamma}{R\Gamma}

\newcommand{\Hom}{\operatorname{Hom}}
\newcommand{\Map}{\operatorname{Map}}

\newcommand{\Ext}{\operatorname{Ext}}

\newcommand{\End}{\operatorname{End}}

\newcommand{\Gal}{\operatorname{Gal}}
\newcommand{\Perf}{\operatorname{Perf}}

\newcommand{\Spa}{\operatorname{Spa}}
\newcommand{\Spf}{\operatorname{Spf}}
\newcommand{\Spec}{\operatorname{Spec}}

\newcommand{\cts}{{\operatorname{cts}}}
\newcommand{\lc}{{\operatorname{lc}}}
\newcommand{\an}{\mathrm{an}}
\newcommand{\hotimes}{\hat{\otimes}}

\newcommand{\et}{{\mathrm{\acute{e}t}}}
\newcommand{\proet}{\mathrm{pro\acute{e}t}}

\newcommand{\qproet}{\mathrm{qpro\acute{e}t}}

\newcommand{\etqcqs}{{\mathrm{\acute{e}t}\text{-}\mathrm{qcqs}}}

\newcommand{\Zar}{\mathrm{Zar}}
\newcommand{\ZAR}{\mathrm{ZAR}}
\newcommand{\AN}{\mathrm{AN}}

\newcommand{\HT}{\operatorname{HT}}
\newcommand{\aeq}{\stackrel{a}{=}}

\newcommand{\wh}{\widehat}
\newcommand{\U}{\O^\times_1}
\newcommand{\bOx}{\overbar{\O}^{\times}}
\newcommand{\bOpx}{\overbar{\O}^{+\times}}

\newcommand{\Pic}{\operatorname{Pic}}
\newcommand{\uP}{\underline{\Pic}}
\newcommand{\NS}{\mathrm{NS}}
\newcommand{\LSD}{\mathrm{LSD}}
\newcommand{\Spd}{\mathrm{Spd}}
\newcommand{\wt}{\widetilde}
\newcommand{\perf}{\mathrm{perf}}
\renewcommand{\diamond}{\diamondsuit}

\newcommand{\twolim}{2\text{-}\varinjlim}
\newcommand{\rk}{\mathrm{rk}}
\newcommand{\Char}{\operatorname{char}}
\renewcommand{\lim}{\varprojlim}
\newcommand{\Rlim}{R\!\lim}
\newcommand{\cH}{{\ifmmode \check{H}\else{\v{C}ech}\fi}}
\newcommand{\tf}{[\tfrac{1}{p}]}

\newcommand*\isomarrow{%
	\xrightarrow{\raisebox{-0.35em}{\smash{\ensuremath{\sim}}}}
}
\newcommand*\isomfrom{%
	\xleftarrow{\raisebox{-0.35em}{\smash{\ensuremath{\sim}}}}
}  

\tikzset{
	labelrotate/.style={anchor=south, rotate=90, inner sep=.5mm}} 
\tikzset{
	labelrotatep/.style={anchor=north, rotate=90, inner sep=.75mm}}

\makeatletter
\newcommand{\doublewidetilde}[1]{{%
		\mathpalette\double@widetilde{#1}%
}}
\newcommand{\double@widetilde}[2]{%
	\sbox\z@{$\m@th#1\widetilde{#2}$}%
	\ht\z@=.85\ht\z@
	\widetilde{\box\z@}%
}
\def\mathcenterto#1#2{\mathclap{\phantom{#1}\mathclap{#2}}\phantom{#1}}
\let\old@widetilde\doublewidetilde
\def\widetildeto#1#2{\mathcenterto{#2}{\doublewidetilde{\mathcenterto{#1}{#2\,}}}}
\makeatother
\newcommand{\wtt}[1]{\widetildeto{I}{#1}}

\begin{document}
\title[Line bundles on perfectoid covers]{Line bundles on perfectoid covers:\\case of good reduction}
\author{Ben Heuer}
\date{}

\maketitle
\begin{abstract}
	We study Picard groups and Picard functors of perfectoid spaces which are limits of rigid spaces. For sufficiently large covers that are limits of rigid spaces of good reduction, we show that the Picard functor can be represented by the special fibre.
	We use our results to answer several open questions about Picard groups  of perfectoid spaces from the literature, for example we show that these are not always $p$-divisible. Along the way, we construct a ``Hodge--Tate spectral sequence for $\G_m$'' of independent interest.
\end{abstract}
\section{Introduction}
The goal of this article is to study line bundles and Picard groups of perfectoid spaces. In particular, we describe first instances of representable Picard functors of perfectoid spaces.

Since their introduction  by Scholze \cite{perfectoid-spaces},
perfectoid spaces have found manifold applications in number theory, algebra and geometry. In many such applications, perfectoid spaces arise naturally as certain limits (``tilde-limits'') of infinite towers of rigid spaces. For example, this happens in the context of $p$-adic Hodge theory \cite{Scholze_p-adicHodgeForRigid} as well as for Rapoport--Zink spaces and Shimura varieties of infinite level \cite{ScholzeWeinstein}\cite{torsion}. It is therefore of general interest to understand the behaviour of line bundles in such infinite perfectoid towers. For this reason, in this article, we set out to launch a first systematic study of the following natural  question:

Throughout, let $K$ be any perfectoid field of residue characteristic $p$, for example $K=\C_p$.
	\begin{Question}\label{q:Pic(X)->Pic(X_infty)-surj}
		Let $X_\infty\sim \varprojlim_{i\in I} X_i$ be a perfectoid tilde-limit of a cofiltered  inverse system of smooth rigid spaces over $K$ with finite transition maps.  Consider the pullback map
		of analytic Picard groups
		\begin{equation}\label{eq:q1}
			\textstyle\varinjlim_{i\in I} \Pic(X_i)\to \Pic(X_\infty).
		\end{equation}
		How far is this from being an isomorphism, i.e.\
		what are its kernel and cokernel? 
	\end{Question}
	\begin{Remark}\label{rem:motivation}
	We are interested in these questions for the following concrete reasons:
	\begin{enumerate}
		
		\item It is natural to ask how much of the many known results about Picard groups of rigid spaces, e.g.\ representability of Picard functors, carries over to more general settings like the pro-\'etale site introduced by Scholze \cite{Scholze_p-adicHodgeForRigid}. This especially concerns the case where $X_\infty$ is a perfectoid object in the pro-\'etale site of a rigid space $X$.
		\item When $A$ is any abelian variety over $K$ and $K$ is a complete algebraically closed field, then the multiplication-by-$p$ maps define a perfectoid tilde-limit $A_\infty\sim \varprojlim_{[p]}A$ by \cite[Theorem~1]{perfectoid-covers-Arizona}. In \cite{heuer-isoclasses}, 
		we will use the results of this article about $\Pic(A_\infty)$ to prove a new pro-\'etale uniformisation theorem for abelian varieties.
		
		\item Let $X$ be a connected smooth proper variety over an algebraically closed complete field $K$ over $\Q_p$. Let $(X_i)_{i\in I}$ be the cofiltered inverse system of connected finite \'etale covers of $X$ together with a choice of lift of some base point of $X$ (cf \cite[\S4.3]{heuer-v_lb_rigid}). This has a perfectoid tilde-limit $X_\infty$ at least when $X$ is a curve \cite[Corollary 1.1]{perfectoid-covers-Arizona}.
		
		The question which vector bundles on $X$ are trivialised on $X_\infty$ is intimately related to Faltings' open question \cite[\S5]{Faltings_SimpsonI} which Higgs bundles correspond to local systems under the $p$-adic Simpson correspondence \cite[\S3]{wuerthen_vb_on_rigid_var}\cite[\S5]{heuer-v_lb_rigid}\cite[\S1.2]{heuer-geometric-Simpson-Pic}.  This is expected to be a very deep question, especially after the recent work of Andreatta \cite{andreatta2024padic} has ruled out what was previously considered the conjectural answer. It is known from  \cite{heuer-geometric-Simpson-Pic} and \cite{HX} that the moduli stack of vector bundles on $X_\infty$ holds the key to this question, but it is currently completely unclear what this looks like. Studying the Picard functor of $X_\infty$ is therefore an important first step in this direction.
		
		\item When instead $K$ has characteristic $p$, one particular case of interest is the perfection $X^\perf\sim\varprojlim_{F}X$ from \cite[III.2.18.(ii)]{torsion}. In this case, \Cref{q:Pic(X)->Pic(X_infty)-surj} becomes:
	\end{enumerate}
\end{Remark}

\begin{Question}\label{q:Pic(X)->Pic(X^perf)-kernel}
	Assume that $K$ is a perfectoid field of characteristic $p$ and let $X$ be a rigid space over $K$.
	As multiplication by $p$ is bijective on $\Pic(X^\perf)$, there is a natural map
	\begin{equation}\label{eq:q2}
		\Pic(X)\tf\to \Pic(X^\perf).
	\end{equation}
	What are its kernel and cokernel?
\end{Question}
	
	In characteristic $p$, there is for proper $X$ a fully faithful functor from $K$-linear representations of the \'etale fundamental group of $X$ to vector bundles on $X^\perf$. As in the $p$-adic case, it is therefore of interest to study the difference between the categories of vector bundles on $X$ and $X^\perf$, potentially leading to a kind of ``$t$-adic Simpson correspondence'' over $\F_p(\!(t)\!)$.
	
	\medskip

	While keeping these motivations in mind, we leave the specific applications to subsequent works, and will in this article focus on studying the above questions in their own right.

\begin{Remark}\label{r:previously-known}
The following has previously been known regarding Questions \ref{q:Pic(X)->Pic(X_infty)-surj} and \ref{q:Pic(X)->Pic(X^perf)-kernel}:
\begin{enumerate}
	\item One case that is well-understood is when the $X_i$ are all affinoid and $X_\infty$ is affinoid perfectoid: In this case, Kedlaya--Liu have shown in \cite{KedlayaLiu-II} that vector bundles on $X_\infty$ correspond to projective $\O(X_\infty)$-modules, and using a result of Gabber--Ramero (see \Cref{p:Gabber-Ramero-5.4.42}), one deduces that $\eqref{eq:q1}$ is an isomorphism in this case.
\item 
The case of $X=\P^n$ has been studied by Dorfsman-Hopkins, who showed in his thesis \cite{dorfsman2019projective} that $\eqref{eq:q2}$ is an isomorphism for $X=\P^n$, extending the case of $n=1$  previously treated by Das \cite{das2016vector}. Dorfsman-Hopkins--Ray--Wear have extended this to smooth projective toric varieties \cite{perfectoid_covers_toric}, where the map \eqref{eq:q2} is also an isomorphism.
\item In the specific setting of \Cref{rem:motivation}.2,  we had previously described in \cite{heuer-geometric-Simpson-Pic} the kernel of the map $\Pic(X)\to \Pic(X_\infty)$: When the rigid analytic Picard functor $\underline{\Pic}_{X}$ is representable, the kernel is the open subgroup of ``topologically torsion'' elements. In this article, we give a different perspective on this phenomenon and also describe $\Pic(X_\infty)$ as well as the Picard functor of $X_\infty$ in some first cases of interest.
\end{enumerate}
\end{Remark}

\subsection{The case of good reduction}

Our first main result answers Questions~\ref{q:Pic(X)->Pic(X_infty)-surj} and \ref{q:Pic(X)->Pic(X^perf)-kernel} in situations of ``good reduction'', i.e.\ when each $X_i$ is the generic fibre of a smooth formal scheme $\mathfrak X_i$ over $\O_K$.
Under additional assumptions which in practice are satisfied e.g.\ by ``universal covers'' such as those in \Cref{rem:motivation}.2, we show that the Picard group of $X_\infty$ can be described in terms of the special fibres of the $\mathfrak X_i$ over the residue field $k$ of $\O_K$:
\begin{Theorem}\label{t:intro-Pic(X_infty)}
			Let $(\mathfrak X_i)_{i\in I}$ be a cofiltered inverse system of smooth qcqs formal $\O_K$-schemes with affine transition maps. Assume that the adic generic fibre $X_\infty$ of $\varprojlim_{i\in I}\mathfrak X_i$ is perfectoid. Let $X_i$ be the generic fibre and $\overbar{X}_i$ the special fibre of $\mathfrak X_i$. Moreover:
		\begin{enumerate}[label=(\roman*)]
			\item If $\Char K=0$, assume that $H^j_{\an}(X_\infty,\O)=0$ for $j=1,2$.
			\item If $\Char K=p$, assume that $H^j_{\proet}(X_{\infty},1+\m\O^+)=0$ for $j=1,2$.
		\end{enumerate}
		Then the pullback $\Pic(X_i)\to \Pic(X_\infty)$ induces a natural  isomorphism
		\[\textstyle\Pic(X_\infty)[\tfrac{1}{p}]=\varinjlim_{i\in I} \Pic(\overbar{X}_i)[\tfrac{1}{p}].\]
\end{Theorem}
For example, if $\Char K=p$, we will explain that condition (ii) holds if the $\mathfrak X_i$ are all proper and the Frobenius endomorphism $F$ acts nilpotently on $H^j(X_{i},\O)=0$ for $j=1,2$. This allows us to answer Question~\ref{q:Pic(X)->Pic(X^perf)-kernel} for ``supersingular'' varieties in the sense of  \cite[\S0]{FakhruddinSupersingular}, i.e.\ if  all Newton polygons of the crystalline cohomology of $X$ have constant slope:

\begin{Corollary}\label{c:Pic(X^perf)-for-X-with-trivial-cohom}
	Let $ X$ be a geometrically connected smooth proper rigid space over a perfectoid field $K$ of characteristic $p$ with good reduction $\overbar{X}$. If $F$ acts nilpotently on $H^j(X,\O)$ for $j=1,2$, then
	\[\Pic( X^\perf)=\Pic(\overbar X)\tf.\]
\end{Corollary}
Note that the Theorem and its Corollary generalise all known cases mentioned in Remark~\ref{r:previously-known}.2.
For example, from our perspective, the reason that $\Pic(\P^{n,\perf})=\Pic(\P^n)\tf$ lies in the fact that $\P^n$ has the same Picard group as its reduction over the residue field.

\medskip

Second, we also give a useful criterion for when the map in \Cref{q:Pic(X)->Pic(X_infty)-surj} is in fact bijective:
\begin{Theorem}\label{t:Picard-bij-intro}
	Let $K$ be a perfectoid field extension of $\Q_p$.
	Let $X_\infty\sim \varprojlim_{i\in I} X_i$ be a perfectoid tilde-limit of qcqs smooth rigid spaces with affinoid transition maps.  Suppose that
	\[\textstyle\varinjlim_{i\in I} H^j_\an(X_i,\O)\to H^j_\an(X_\infty,\O)\]
	is an isomorphism for $j=0,1,2$. Then also the following map is an isomorphism:
	\[\textstyle \varinjlim_{i\in I} \Pic(X_i)\isomarrow \Pic(X_\infty).\]
\end{Theorem}
Together with \Cref{t:intro-Pic(X_infty)}, this indicates that in general, the answer to \Cref{q:Pic(X)->Pic(X_infty)-surj} can be found by looking at the cohomology of $\O$, which is often much easier to understand.
\subsection{The multiplicative Hodge--Tate sequence}
For the proof of Theorems~\ref{t:intro-Pic(X_infty)} and \ref{t:Picard-bij-intro}, we study a certain sheaf $\bOx$ on the $v$-site of rigid and perfectoid spaces. This appeared already in \cite{heuer-v_lb_rigid}, but we now clarify that in some precise sense, it ``computes line bundles on the special fibre''. We use this to prove some cohomological results of independent interest. For example, Scholze has pointed out to us that these can be used to see the following:
\begin{Theorem}\label{t:v-Brauer-intro}
	\begin{enumerate}
	\item If $X$ is a perfectoid space over $K$, then for the map  $\nu:X_{v}\to X_{\et}$, \[R\nu_{\ast}\O^\times_v=\O^\times_{\et}.\]
	\item If $X$ is a smooth rigid space over $K$ with $\Char K=0$, then $\nu_{\ast}\O^\times_v=\O^\times_{\et}$ and 
	\[ R^i\nu_{\ast}\O^\times_v= \Omega_{X_{\et}}^i\{-i\} \quad \text{ for } i\geq 1.\]
	\end{enumerate}
\end{Theorem}
The first part could be regarded as a perfectoid analogue of  Grothendieck's result  that for a scheme $X$, we have $H^i_{\et}(X,\G_m)=H^i_{\mathrm{fppf}}(X,\G_m)$ for $i\geq 0$ \cite[Theorem~11.7]{Grothendieck_BrauerIII}.

The $\{-i\}$ in part 2 is a Breuil--Kisin--Fargues twist. The Leray spectral sequence then produces what one could describe as a ``multiplicative Hodge--Tate spectral sequence''
\[ E_2^{ij}:=\begin{dcases}\begin{rcases}
		H^i_{\et}(X,\O^\times)& \text{ if }j=0\\
		H^i_{\et}(X,\Omega_X^j\{-j\})& \text{ if }j>0	
\end{rcases}\end{dcases}\Rightarrow H^{i+j}_{v}(X,\O^\times).
\]

\subsection{Abelian varieties}
As the second main new case of Question~\ref{q:Pic(X)->Pic(X_infty)-surj}, we apply Theorem~\ref{t:intro-Pic(X_infty)} to $X_i=B$ an abelian variety with good reduction $\overbar{B}$ and its $p$-adic universal cover 
\[\textstyle\wt B:=\varprojlim_{[p]}B.\] If $K$ is algebraically closed, this is a perfectoid space \cite{perfectoid-covers-Arizona}. Theorem~\ref{t:Pic-functor-for-wtB-intro} then implies that
\[ \Pic(\wt B)=\Pic(\overbar{B})\tf.\]

Via the rigid Albanese varieties of Hansen--Li \cite[\S4]{HansenLi_HodgeSymmetry}, one can deduce in great generality that $p$-adically close line bundles on proper rigid spaces of good reduction become isomorphic on large pro-finite covers, i.e.\ the kernel of \eqref{eq:q1} contains an open subgroup for such covers.
\subsection{Picard functors}
In rigid analytic geometry, the question of representability of Picard functors has a long history shaping the subject, see \cite[\S1]{heuer-diamantine-Picard} for an introduction and overview.

For proper smooth rigid spaces $X$ in characteristic $0$, we showed in  \cite[Theorem~1.1]{heuer-diamantine-Picard} that the  Picard functor of $X$ defined on perfectoid test objects is the diamondification of the rigid Picard functor. We used this to characterise line bundles on $X$ that are trivialised by the universal cover $X_\infty$. In the case of good reduction, we can now complement these results by describing the Picard functor of  the universal cover $X_\infty$ itself.

Namely, we deduce from our results that a relative version of Theorem~\ref{t:intro-Pic(X_infty)} holds if the $\mathfrak X_i$ are proper: Let $\pi:X_\infty\to \Spa(K)$ be the structure map and consider the Picard functor \[\uP_{X_\infty}:=R^1\pi_{\et}\G_m:\Perf_{K,\et}\to \mathrm{Ab},\quad T\mapsto \Pic(T\times X_\infty)/\Pic(T)\] of $X_\infty$ defined on perfectoid test objects. Then one can make precise the idea that this is represented by the colimit of the (usual algebraic) Picard varieties of the special fibres of the $X_i$ over $k$. For example, in the case of abeloids, this means the following:
\begin{Theorem}\label{t:Pic-functor-for-wtB-intro}
	Let $B$ be an abeloid variety with good reduction. Let $\wt B=\varprojlim_{[p]}B$ be its $p$-adic universal cover. Let $\overbar{B}$ be the special fibre over $k$ and let $\uP_{\overbar{B}}$ be its Picard variety over $k$. Then there is a natural isomorphism of $v$-sheaves
	\[ \uP_{\wt B}=(\uP_{\overbar{B}})^\diamond\tf,\]
	where $(\uP_{\overbar{B}})^\diamond$ is the sheafification of the functor sending $(S,S^+)\in \Perf_K$ to $\uP_{\overbar{B}}(S^+/\mathfrak m)$.
\end{Theorem}
This is the first instance where the ``Picard variety'' of a perfectoid object in the pro-\'etale site can be described explicitly. We also give a version of  \Cref{t:Picard-bij-intro} for Picard functors (\Cref{t:Picard-bij} below).
Together, these
give a first idea of what kind of geometric objects represent Picard functors of perfectoid spaces that are limits of proper rigid spaces.

\medskip

Finally, in Section 6, we use our results to answer a few general open questions on line bundles on perfectoid covers that have been asked in the literature. For example, the case of abelian varieties shows that the morphism \eqref{eq:q2} is typically neither injective nor surjective, and that Picard groups of perfectoid spaces are not always $p$-divisible.

\subsection*{Acknowledgements}
We would like to thank Peter Scholze for very helpful discussions and comments on an earlier version,   in particular for suggesting Theorem~\ref{t:v-Brauer} and its proof and for suggesting the argument for Lemma~\ref{l:bOx-p-divisible} in characteristic $p$.
 We thank Urs Hartl and Werner L\"utkebohmert for explaining to us how to prove Lemma~\ref{l:BL8-11} for general non-archimedean base fields, and we thank Ian Gleason and Peter Scholze for both pointing out an inaccuracy in an earlier version of Lemma~\ref{l:sheaf-assoc-to-special-fibre}.
 
 We moreover thank Johannes Ansch\"utz,  Gabriel Dorfsman-Hopkins, David Hansen, Lucas Mann, Peter Wear and Annette Werner for very helpful conversations. 

This work was funded by the Deutsche Forschungsgemeinschaft (DFG, German Research Foundation) under Germany's Excellence Strategy-- EXC-2047/1 -- 390685813. The author was supported by the DFG via the Leibniz-Preis of Peter Scholze.

\section{The sheaf $\bOx$}
Let $K$ be a perfectoid field of residue characteristic $p$, as in the introduction. We denote by $\O_K$ the ring of integers, $\mathfrak m$ the maximal ideal, $\varpi\in \mathfrak m$ some pseudo-uniformiser, and $k$ the residue field. In this section, we will sometimes relax the condition on $K$ and allow it to be any non-archimedean field over $\Z_p$. We will indicate this in the beginning of the subsection.
\subsection{Recollections on line bundles in various topologies}
In the following, we will freely work with diamonds in the sense of \cite[\S11]{etale-cohomology-of-diamonds}\cite[\S8]{ScholzeBerkeleyLectureNotes}.
We will almost exclusively deal with diamonds $X$ that arise from an analytic adic space over a perfectoid field $K$. In this situation, we have a better handle over the structure sheaf $\O$ on $X$, and thus also on the sheaf of units $\O^\times$ which we need to describe line bundles. 

Recall that analytic adic spaces can be regarded as diamonds via the ``diamondification''
\[ \{\text{analytic adic spaces over }K\}\to \LSD_K,\quad X\mapsto X^\diamond\]
of \cite[\S15]{etale-cohomology-of-diamonds},
where $\LSD_K$ denotes the category of locally spatial diamonds over $\Spd(K)$. Due to the fixed structure map to $K$, we can regard $X^\diamondsuit$ as a sheaf on $\Perf_{K,v}$, the site of perfectoid algebras over $K$ with the $v$-topology. On this, we can simply describe $X^\diamond$ as
\[ X^{\diamond}:(S,S^+)\mapsto X(S,S^+).\]

If  $K$ is of characteristic $0$, then the diamondification functor is fully faithful on the subcategory of semi-normal rigid spaces over $K$. This follows from a result of Kedlaya--Liu, see \cite[Proposition~10.2.3]{ScholzeBerkeleyLectureNotes}\cite[Theorem 8.2.3]{KedlayaLiu-II}.

In characteristic $p$, the functor instead sends a rigid space $X$ to its perfection
\[X^\perf:=\textstyle\varprojlim_FX,\]
where the limit is taken in $v$-sheaves. In particular, it is in general far from fully faithful.

Let now $X$ be an analytic adic space over $K$. We will usually assume that $X$ is either a perfectoid space or a rigid space over $K$.
We consider the analytic site $X_{\an}$ and the (small) \'etale site $X_{\et}$ in the sense of \cite[Definition 8.2.19]{KedlayaLiu-II}. These are related by a morphism 
\[ r:X_{\et}\to X_{\an}.\]
By \cite[Lemma 15.6]{etale-cohomology-of-diamonds}, diamondification induces an equivalence
\[X_{\et}=X^\diamond_{\et}\]
that allows us to freely pass back and forth between $X$ and its associated diamond: This is harmless as long as we remember that if $X$ is a rigid space of characteristic $p$, or one of characteristic $0$ that fails to be semi-normal, this changes the structure sheaf $\O$.

If $X$ is a perfectoid space, we also consider the $v$-site $X_v$ and the pro-\'etale site $X_{\proet}$ from \cite[Definition 8.1]{etale-cohomology-of-diamonds}. Let
\[ \nu:X_{\proet}\to X_{\et}\]
be the natural morphism of sites.

We now briefly recall how the notions of line bundles (i.e.\ of torsors under $\O^\times$) are related for these different topologies. We refer to \cite[\S2]{heuer-v_lb_rigid} for a more detailed discussion.

If $X$ is a rigid space, then by \'etale descent, the category of line bundles on $X_{\et}$ is equivalent to that on $X_{\an}$. On perfectoid spaces, Kedlaya--Liu prove a stronger statement:
\begin{Theorem}[{\cite[Theorem 3.5.8]{KedlayaLiu-II}, \cite[Lemma 17.1.8]{ScholzeBerkeleyLectureNotes}}]\label{t:KL-line-bundles}
	Let $X$ be a perfectoid space over $K$. Any line bundle on $X$ for the $v$-topology is already locally trivial in the analytic topology.
\end{Theorem}

The analogous statement  is not true for rigid spaces over $K$: By \cite[Theorem~1.3]{heuer-v_lb_rigid}, there are usually many more line bundles in the pro-\'etale topology than in the analytic topology.

Kedlaya--Liu further prove that Kiehl's Theorem A holds in great generality:
\begin{Theorem}[{\cite[Theorem 8.2.22.(b)]{KedlayaLiu-rel-p-p-adic-Hodge-I}}]\label{t:KL-Kiehl-A}
	Let $X$ be a (sheafy) affinoid adic space, then line bundles on $X$ for the analytic topology are equivalent to projective $\O(X)$-modules of rank 1.
\end{Theorem}

As mentioned in the introduction, we can use Theorem~\ref{t:KL-Kiehl-A} to answer Question~\ref{q:Pic(X)->Pic(X_infty)-surj} in the affinoid case (under mild technical assumptions), by the following result of Gabber--Ramero:
\begin{Proposition}[{\cite[Corollary~5.4.42]{GabberRamero}}]\label{p:Gabber-Ramero-5.4.42}
	Let $R$ be a commutative ring, $\varpi\in R$ a non-zero-divisor, $\widehat{R}$ the $\varpi$-adic completion. Assume $(R,(\varpi))$ is henselian. Then the base change
	\[R[\tfrac{1}{\varpi}]\text{-}\mathbf{Mod}\to \widehat{R}[\tfrac{1}{\varpi}]\text{-}\mathbf{Mod} \]
	induces a bijection between the isomorphism classes of finite projective modules.
	Hence, if $(X_i)_{i\in I}$ is a cofiltered inverse system of uniform affinoid adic spaces over $K$ and $X$ is an affinoid adic space with compatible maps $X\to X_i$ such that $\O(X)=(\varinjlim_{i\in I} \O^+(X_i))^\wedge[\tfrac{1}{\varpi}]$,
	\[\textstyle \Pic_{\an}(X)=\varinjlim_{i\in I} \Pic_{\an}(X_i).\]
\end{Proposition}
\begin{proof}
	The first part is \cite[Corollary~5.4.42]{GabberRamero}. The second part follows by Theorem~\ref{t:KL-Kiehl-A}.
\end{proof}

We recall from \cite[Definition 2.4.2]{huber2013etale} the notion of tilde-limits: Let $(X_i)_{i\in I}$ be a cofiltered inverse system of adic spaces with qcqs transition maps. An adic space $X_\infty$ with compatible maps $(f_i:X_\infty\to X_i)_{i\in I}$ is called the tilde-limit and we write
\[ \textstyle X_\infty\sim \varprojlim_{i\in I} X_i\]
if the $f_i$ induce a homeomorphism $|X_\infty|=\varprojlim |X_i|$ and there is a cover of $X_\infty$ by affinoid open subspaces $U_\infty$ for which $\varinjlim_{U}\O(U)\to \O(U_\infty)$
has dense image, where the colimit is over all $i\in I$ and all affinoid opens $U\subseteq X_i$ through which $U_\infty\to X_i$ factors. If all $X_i$ and $X_\infty$ are affinoid, we write  $X_\infty\approx \varprojlim_{i\in I} X_i$
if already $\varinjlim \O(X_i)\to \O(X_\infty)$ has dense image. Under mild technical assumptions, any tilde-limit is locally of this form:
\begin{Lemma}\label{l:tilde-limit-locally-approx}
	Let $X_\infty\sim \varprojlim_{i\in I} X_i$ be a tilde-limit of a cofiltered inverse system of quasi-compact adic spaces over $K$. Assume that each $X_i$ has an open cover by affinoid spaces $U_i$ such that $U_i\times_{X_i}X_j$ is affinoid for all $j\geq i$. Then there is $i\in I$ and an open cover of $X_i$ by affinoids $U$ such that $U\times_{X_i}X_\infty\approx \varprojlim_{j\geq i} U\times_{X_i}X_j$ is  a tilde-limit of affinoid spaces. 
	
	If $X_\infty$ is perfectoid, we can moreover arrange for $U\times_{X_i}X_\infty$ to be affinoid perfectoid.
\end{Lemma}
\begin{proof}
	Let $x\in X_\infty$ be any point. The tilde-limit condition asserts that we can find an affinoid open neighbourhood $x\in U$ such that 
	\[ \varinjlim_{(V,i)} \O(V)\to \O(U)\]
	has dense image, where $(V,i)$ ranges through $i\in I$ and all affinoids $V\subseteq X_i$ through which $U\to X\to X_i$ factors. Let $x\in U'\subseteq U$ be any affinoid (respectively, affinoid perfectoid) open subspace. Since rational subspaces of $U$ form a basis for the topology, we can find a rational open subspace $V$ of $U$ such that $x\in V\subseteq U'$ and such that $V$ is contained in the pullback of one of the affinoid $U_i\subseteq X_i$ in the statement of the Lemma. Since $V$ is rational in $U$, it  is in particular also rational in $U'$ (hence affinoid perfectoid). Since $|X_\infty|=\varprojlim_{i\in I} |X_i|$, we may then further replace $V$ by the pullback of some rational open $V_i$ in $U_i$. Set $V_j:=V_i\times_{X_i}X_j$, this is affinoid because  $U_i\times_{X_i}X_j$ is affinoid by assumption and the pullback of any rational subspace along a morphism of affinoid adic spaces is rational. This combines with \cite[Remark 2.4.3.(ii)]{huber2013etale} to show that 
	\[\textstyle V\approx \varprojlim_{j\geq i}X_j\times_{X_i}X_j\]
	is an (affinoid perfectoid) tilde-limit. Letting $x$ vary and using that $X$ is quasi-compact by \cite[Tag 08ZV]{StacksProject}, this proves the lemma.
\end{proof}

\subsection{Recollection on replete topoi}
We now recall some background from \cite[\S3]{bhatt-scholze-proetale} and \cite[\S10]{ScholzeBerkeleyLectureNotes} on replete topoi.
\begin{Definition}[{\cite[Definition 3.1.1]{bhatt-scholze-proetale}}]
	A topos $\mathfrak X$ is called replete if for any inverse system $(F_n)_{n\in\N}$ in $\mathfrak X$ with surjective transition maps, the projection $\varprojlim_{n\in \N} F_n\to F_m$ is surjective for all $m\in \N$. We say that a site is replete if its associated topos is replete.
\end{Definition}
The following example of replete sites is well-known (see \cite{ScholzeBerkeleyLectureNotes} before Proposition~10.4.3).
\begin{Lemma}\label{l:qproet-replete}
	Let $X$ be a diamond. Then $X_{\qproet}$ and $X_{v}$ are replete.
\end{Lemma}
This can be seen exactly like in \cite[Example 3.1.7]{bhatt-scholze-proetale}. Since there does not seem to be a proof in the literature, we present the argument here for the convenience of the reader.
\begin{proof}
	Let $(F_n)_{n\in\N}$ be a system of sheaves on $X_v$. Let $Y\to X$ be an object of $X_v$ and let $s\in F_0(Y)$.
	Since any diamond admits a quasi-pro-\'etale cover by a perfectoid space, we may without loss of generality assume that $Y$ is affinoid perfectoid. Then there exists a $v$-cover $Y_{1}\to Y$ by some perfectoid space such that $s$ lifts to $F_{1}(Y_1)$. We may assume that $Y_{1}$ is affinoid perfectoid: Covering $Y_{1}$ by affinoid perfectoid opens $U_i$, we see by the ``quasi-compactness'' condition in the definition of $v$-covers that finitely many of the $U_i$ cover $Y_{1}$, and their disjoint union is affinoid perfectoid. 
	
	Iterating this process, we get a system $\dots \to Y_{2}\to Y_{1}\to Y$ of affinoid perfectoid $v$-covers with a compatible system of lifts of $s$ to $F_m(Y_m)$. Then there is a unique affinoid perfectoid tilde-limit $Y_\infty\sim\varprojlim Y_m$ by \cite[Proposition 6.4]{etale-cohomology-of-diamonds}. By the tilde-limit property, $Y_\infty\to Y_1$ is still a $v$-cover, and by construction $s$ has a preimage under $\varprojlim  F_n(Y_\infty)\to F_1(Y_\infty)$.
	
	The quasi-pro-\'etale case is similar: Here we first use that by \cite[Lemma 7.18]{etale-cohomology-of-diamonds}, there is a quasi-pro-\'etale cover  $\wt Y\to Y$ by a strictly totally disconnected space, and replacing $Y$ by $\wt Y$ we may thus assume that $Y$ is strictly totally disconnected. 
	
	Any quasi-pro-\'etale cover $Y_1\to Y$ is then pro-\'etale, and $Y_1$ is itself affinoid perfectoid and  strictly totally disconnected by \cite[Lemma 7.19]{etale-cohomology-of-diamonds}. Again, we find an inverse system
	\[\dots \to Y_{n}\to \dots \to Y_{1}\to Y\]
	with compatible lifts of $s\in F(Y)$ to $F(Y_n)$.
	By  \cite[
	Proposition 7.10 and Lemma 7.11]{etale-cohomology-of-diamonds}, there is an affinoid perfectoid tilde-limit $Y_\infty\sim\varprojlim_n Y_{n}\to Y$ that is still a pro-\'etale cover.
	We thus see that $s$ is in the image of  $\varprojlim_nF_n(Y_\infty)\to F_1(Y_{\infty})$, as we wanted to see.
\end{proof}
\begin{Corollary}\label{c:replete-derived-complete}
	On either $X_v$ or $X_{\qproet}$,
	for an inverse system of abelian sheaves $(F_n)_{n\in\N}$ with surjective transition maps we have
	\[ \textstyle\Rlim_nF_n=\varprojlim_nF_n.\]
\end{Corollary}
\begin{proof}
	By \cite[Proposition 3.1.10]{bhatt-scholze-proetale}, this follows from Proposition~\ref{l:qproet-replete}.
\end{proof}

\subsection{Definition of $\bOx$ and basic properties}
In order to study line bundles in inverse limits, the sheaf $\bOx$ used in \cite[\S2.3]{heuer-v_lb_rigid} will be crucial for us.
We begin by recollecting its definition and some of its basic properties. For this we can more generally let $K$ be any non-archimedean field over $\Z_p$, not necessarily perfectoid.
\begin{Definition}\label{d:bOx}
	Let either $X$ be a rigid space over $K$ and $\tau=\an$, $\et$, or if $\Char K=0$ also $\proet$ in the sense of \cite[\S3]{Scholze_p-adicHodgeForRigid}. Or let $X$ be a perfectoid space over $K$ and $\tau$ one of $\an,\et,\proet,v$. Let \[\O^\times_{1\tau}:=1+\m\O^+_\tau\]
	be the sheaf of principal units in $\O^\times_\tau$. We then let $\bOx_{\tau}$ be the sheaf on $X_\tau$ defined by 
	\[ \bOx_{\tau}:=(\O^\times_{\tau}/\O^\times_{1\tau})\tf.\]
	Here the $\tf$ means that we formally invert $p$, i.e.\ form the colimit over the maps $x\mapsto x^p$.
\end{Definition}
We note that at first glance, this definition might appear to differ from the one used in \cite[\S2.3]{heuer-v_lb_rigid}\cite[\S2.3]{heuer-diamantine-Picard}, due to the additional $\tf$ in the definition. However, the definitions actually agree due to the following lemma: Only for the analytic topology (which we did not consider in the aforementioned settings) do we need to invert $p$ to get the correct statements.
\begin{Lemma}\label{l:bOx-p-divisible}
	Let $\tau$ be the \'etale topology, or any finer topology. Then
	\[ \O^\times_{\tau}/\O^\times_{1\tau}=\O^\times_{\tau}/\O^\times_{1\tau}\tf.\]
\end{Lemma}
\begin{proof}
	 If $\Char K=0$, this holds by the Kummer sequence \cite[Lemma~2.16]{heuer-v_lb_rigid}. If $\Char K=p$, a variant of this argument still works: Let $U\in X_{\tau}$ be affinoid and let $a\in \O^\times(U)$. For any $n\in \N$, consider the finite \'etale cover $U'\to U$ defined by the equation
	 \[ f:=X^p+\varpi^nX+a=0.\]
	 Then $X\in \O^\times(U')$, and we can write $X^p(1-\varpi^nX^{1-p})=a$.
	 Using the Newton polygon of $f$ at any point of $U'$, we see that that for $n\gg0$, we have $(1-\varpi^nX^{1-p})\in 1+\m\O^+_{\tau}(U')$. Consequently, the image of $X$ is a $p$-th root of $a$ inside $\O^\times_{\tau}/\O^\times_{1\tau}(U')$. 
\end{proof}

The following Proposition is the main reason why $\bOx$ is useful for us, and we used this sheaf in \cite[\S2]{heuer-v_lb_rigid} and \cite[\S2.3]{heuer-diamantine-Picard} for essentially the same reason: In many situations, its cohomology commutes with forming the kinds of limits we are interested in: 
\begin{Proposition}[{\cite[Corollary~3.20]{heuer-diamantine-Picard}}]\label{p:cohom-comparison}
	Let $X_\infty\sim\varprojlim_{i\in I} X_i$ be a perfectoid tilde-limit of sheafy qcqs adic spaces over $K$. Assume that each $X_i$ admits a cover by affinoid spaces $U_i$ such that $U_i\times_{X_i}X_j$ is affinoid for all $j\geq i$.  
	Then for all $j\geq 0$,
	\[\textstyle H^j_{\et}(X_\infty,\overbar{\O}^{\times})=\varinjlim_{i\in I} H^j_{\et}(X_i,\bOx).\]
	The same holds for $\bOx$ replaced by $\O^+/\varpi$  for any pseudo-uniformiser $\varpi\in K$.
\end{Proposition}
\begin{proof}
	The cited \cite[Corollary~3.20]{heuer-diamantine-Picard} shows this under the additional assumption that there is $i\in I$ and a cover of $X_i$ by affinoid open subspaces $U\subseteq X_i$ such that $U\times_{X_i}X_\infty\approx \varprojlim_{j\geq i} U\times_{X_i}X_j$ is an affinoid perfectoid tilde-limit of affinoid adic spaces. By \Cref{l:tilde-limit-locally-approx}, our assumption on affinoid transition maps implies that this is the case.
	
	While this is not stated explicitly, we emphasize that \cite[Corollary~3.20]{heuer-diamantine-Picard} does apply in characteristic $p$: In fact, the short proof only uses \cite[Lemma~3.10,~3.13]{heuer-diamantine-Picard} which explicitly allow characteristic $p$, as well as  \cite[Lemma 3.8]{heuer-diamantine-Picard}, which clearly holds in characteristic $p$.
\end{proof}

\begin{Corollary}\label{c:bOx-for-rigid-proet}
	Let $X$ be a perfectoid space over $K$, or assume $\Char K=0$ and let $X$ be a rigid space. Then for  $u:X_{\proet}\to X_{\et}$ we have $u^{\ast}\bOx_{\et}=\bOx_{\proet}$ and $Ru_{\ast}\bOx_{\proet}=\bOx_{\et}$.
\end{Corollary}
\begin{proof}
	If $X$ is rigid, this follows from Proposition~\ref{p:cohom-comparison} via \cite[Lemma 3.16, Corollary~3.17.(i)]{Scholze_p-adicHodgeForRigid}. If $X$ is perfectoid, it also follows from Proposition~\ref{p:cohom-comparison} via \cite[Proposition 8.5.(i)-(ii)]{etale-cohomology-of-diamonds}.
\end{proof}

On the other hand, we can describe $\bOx$ explicitly in terms of $\O^\times$, at least on affinoid objects. This is easiest if $\Char K=0$, due to the $p$-adic exponential: (see \cite[\S2]{heuer-v_lb_rigid} for details)
\begin{Proposition}[{\cite[Lemmas 2.18 and 2.22]{heuer-v_lb_rigid}}]\label{p:bOx}
	Assume that $\Char K=0$. Let $X$ be a rigid or perfectoid space over $K$. Let $\tau$ be as in Definition~\ref{d:bOx}. Then the $p$-adic exponential $\exp\colon 2p\O^+\to 1+2p\O^+$ defines a short exact sequence on $X_{\tau}$
	\[ 0\to \O_{\tau}\xrightarrow{\exp} \O_{\tau}^\times\tf\to \bOx_{\tau}\to 1.\]
	In particular, we have $\nu_{\ast}\bOx_v=\bOx_{\et}$ and $R^1\nu_{\ast}\bOx_v=0$.
\end{Proposition}
In characteristic $p$, we do not have an exponential sequence, but at least in the perfectoid case we instead have a logarithm sequence, well-known from $p$-adic Hodge theory:
\begin{Lemma}
Let $X$ be a perfectoid space over a perfectoid field of characteristic $p$. Then there is a short exact sequence on $X_{\proet}$:
\[ 0\to \Q_p(1)\to \O^\times_{1}\xrightarrow{\log\sharp} \O^\sharp\to 0.\]
\end{Lemma}

As $\O$ is acyclic on affinoid rigid and affinoid perfectoid spaces (in fact in much greater generality by \cite[Theorem 8.2.22]{KedlayaLiu-rel-p-p-adic-Hodge-I}), Proposition~\ref{p:bOx} allows us to describe the sections of $\bOx$ on affinoid objects in characteristic $0$. With some more work, this is possible in general:
\begin{Lemma}\label{l:bOx-in-arbitrary-char}
	Let $\tau$ be any one of the topologies used in Definition~\ref{d:bOx}.
	\begin{enumerate}
		\item 
		Let $U$ be an affinoid perfectoid space. Then the following map is surjective:
		\[ \O^\times_{\tau}(U)\to \bOx_{\tau}(U).\]
		In fact, this holds for any quasi-compact adic space $U$ with $H^1_\tau(U,\O^+/\varpi)\aeq 0$.
		\item Let $U$ be an affinoid rigid space. Then the following map is surjective:
		\[ \O^\times_{\tau}\tf(U)\to \bOx_{\tau}(U).\]
		\item In either case, $r_{\ast}\bOx_{\et}=\bOx_{\an}$ where $r:U_{\et}\to U_{\an}$ is the natural morphism of sites, and we have natural isomorphisms
		\[ \Pic_{\an}(U)\tf\isomarrow H^1_{\an}(U,\bOx_{\an})\isomarrow H^1_{\et}(U,\bOx_{\et}).\]
	\end{enumerate}
\end{Lemma}
\begin{Remark}
	The affine torus $U=\Spa(K\langle X^{\pm1}\rangle)$ shows that $\O^\times_{\an}/\O^\times_{1\an}\to r_{\ast}(\O^\times_{\et}/\O^\times_{1\et})$ need not be surjective, so we need to invert $p$ in the definition of $\bOx_{\an}$ for part 3 to hold.
\end{Remark}
\begin{proof}
	We begin with the first statement:
	Since $\O^+=\varprojlim_{n\in\N} \O^+/\varpi^n$, it follows that  $\O^{+\times}=\varprojlim_{n\in\N} \O^{+\times}/(1+\varpi^n\O^+)$, from which we see by a 5-Lemma argument that
	\[\O^\times=\varprojlim_{n\in\N} \O^{\times}/(1+\varpi^n\O^+)\]
	(here and in the following, we omit $\tau$ from notation).
	Consequently, it suffices to prove that for $x\in \O^\times/\O^\times_{1}(U)$, we can find a system of compatible lifts of $x$ along
	\[\O^{\times}/(1+\varpi^n\O^+)(U)\to \O^\times/(1+\mathfrak m\O^+)(U).\]
	Since it suffices to prove the statement for $\tau$ finer or equal to $\et$, and since $U$ is perfectoid, the sheaf $\O^\times/\O^\times_{1}$ is already $p$-divisible, so this implies the result for $\bOx$.
	
	To prove the lifting statement, we argue inductively:  Using that 
	\[\O^\times/(1+\mathfrak m\O^+)=\varinjlim_{\epsilon\to 0} \O^\times/(1+\varpi^{\epsilon}\O^+),\]
	we see that $x$ lifts to $x'\in\O^\times/(1+\varpi^{\epsilon}\O^+)(U)$ for some $\epsilon>0$. After replacing the pseudo-uniformiser $\varpi$ by $\varpi^\epsilon$, let us for simplicity of notation assume $\epsilon=1$. We now use a standard lifting argument: For any $1>\delta>0$, consider the morphism of short exact sequences
	\[
	\begin{tikzcd}[column sep = 0.8cm]
		0 \arrow[r] & \O^+/\varpi \arrow[d, "\cdot \varpi^{\delta}"] \arrow[rr,"x\mapsto 1+\varpi x"] && \O^{\times}/(1+\varpi^{2}\O^+) \arrow[d] \arrow[r] & \O^\times/(1+\varpi \O^+) \arrow[d] \arrow[r]&0 \\
		0 \arrow[r] & \O^+/\varpi  \arrow[rr,"x\mapsto 1+x \varpi^{1-\delta}"]                                              && \O^\times/(1+\varpi^{2-\delta}\O^+) \arrow[r]    & \O^\times/(1+\varpi^{1-\delta}\O^+) \arrow[r]  &0
	\end{tikzcd} \]
	Since $H^1_{\tau}(U,\O^+/\varpi)\aeq 0$, this shows that the image of $x'$ in $\O^\times/(1+\varpi^{1-\delta}\O^+)(U)$ can be lifted to $x''\in \O^\times/(1+\varpi^{2-\delta}\O^+)(U)$. The image of $x''$ in  $\O^\times/(1+\varpi^{2-2\delta}\O^+)$ can now be lifted further. Continuing inductively, this gives the desired system of compatible lifts. 
	
	Parts 2 and 3 are clear if $\Char K=0$ by the exponential sequence and acyclicity of $\O$. 
	
	In characteristic $p$, let $U_\infty:=U^\perf$ and consider the commutative diagram
	\[
	\begin{tikzcd}
		\bOx(U) \arrow[r]                   & \bOx(U_\infty)                                 \\
		\O^\times(U)/\U(U) \arrow[u] \arrow[r]& \O^\times(U_\infty)/\U(U_\infty). \arrow[u]
	\end{tikzcd}\]
	We claim that  all its arrows become isomorphisms after inverting $p$: For the morphism on the right, this holds by the first part. For the bottom map, this holds by \cite[Lemma 3.10]{heuer-diamantine-Picard}. %
	 The top map is an isomorphism by Proposition~\ref{p:cohom-comparison} for $j=0$: Here we use that the absolute Frobenius on $\O_X$ can on $\O^\times$ be identified with multiplication by $p$, and thus the same holds for its quotient $\bOx$. 
	Together, this implies that the left map becomes an isomorphism.
	
	Consequently, we have $\bOx_\tau(U)=\O^\times_\tau(U)/\O^\times_{1\tau}(U)\tf$ for any one of the various $\tau$. It follows from this that $\bOx_\an=r_{\ast}\bOx_{\et}$.

	For the statement about $\Pic(U)$, we first consider affinoid perfectoid $U$: Here we note that the above lifting argument, but applied in degree 1 and applied to $U_{\proet}$, also shows that
	\[ H^1_{\proet}(U,\O^\times)=\varprojlim_nH^1_{\proet}(U,\O^\times/(1+\varpi^n\O^+))\isomarrow H^1_{\proet}(U,\bOx)\]
	is an isomorphism. The statement for perfectoid $U$ thus
	follows from the fact that we have $H^1_{\an}(U,\O^\times)=H^1_{\proet}(U,\O^\times)$ by Theorem~\ref{t:KL-line-bundles}. In particular, this implies directly that we have $R^{1}(r\circ\nu)_{\ast}\bOx=0$ where $r\circ\nu:U_{\proet}\to U_{\et}\to U_{\an}$. 
	
	The statement for rigid $U$ in characteristic $p$ now follows from this and the commutative diagram:
	\[
	\begin{tikzcd}
		{H^1_{\tau}(U,\bOx)} \arrow[r, "\sim"] & {H^1_{\tau}(U_\infty,\bOx).}\\
			\Pic_{\an}(U)\tf \arrow[r, "\sim"] \arrow[u] & \Pic_{\an}(U_\infty)\tf \arrow[u, "\sim"labelrotatep]
	\end{tikzcd}\]
	The bottom arrow is an isomorphism by Proposition~\ref{p:Gabber-Ramero-5.4.42}. The top map is an isomorphism by Proposition~\ref{p:cohom-comparison}.
	The right map is an isomorphism by the above, thus so is the left one.
\end{proof}
For perfectoid $X$, a variant of part 3 also holds in higher degree, even before inverting $p$.
\begin{Lemma}\label{l:1+mO-acyclic-on-aff-perf}
	Let $X$ be affinoid perfectoid over $K$, then for $j\geq 1$ and $\tau=v,\proet$ we have $H^j_{\tau}(X,\O^\times_1)=1$.
	In particular, we have $H^j_{\tau}(X,\O^\times)=H^j_{\tau}(X,\bOx)$.
\end{Lemma}
\begin{proof}
	The case of $j=1$ follows from Lemma~\ref{l:bOx-in-arbitrary-char}.1 and 3, here we use that in the proof of part 3 we have seen by a lifting argument that in fact $H^1_{\et}(X,\O^\times)=H^1_{\et}(X,\bOx)$.

	For $j\geq 2$, recall that we have $H^i_{\tau}(X,\O)=0$ for $i\geq 1$ by \cite[Propositions~8.5, 8.8]{etale-cohomology-of-diamonds}.
	If $\Char K=p$, it follows from the Artin--Schreier sequence that $H^j_{\tau}(X,\F_p)=0$ for $j\geq 2$, thus $H^j_{\tau}(X,\Z_p)=0$ by Corollary~\ref{c:replete-derived-complete}. The logarithm sequence implies that  $H^j_{\tau}(X,\O^\times_1)=0$ for $j\geq 2$. The case of $\Char K=0$ follows from the sequence $0\to \Z_p(1)\to   \O^{\times\flat}_1\xrightarrow{\sharp} \O^\times_1\to 1$.
\end{proof}
\subsection{A multiplicative Hodge--Tate spectral sequence}\label{s:mult-HT}
As pointed out to us by Peter Scholze, the above results imply the following improvement of the last part of Proposition~\ref{p:bOx}: This will not be needed in the following, but we think it is interesting in its own right and therefore worth recording. We begin with some notation:
\begin{Definition}
	Let $\theta:W(\O_{K^\flat})\to \O_K$ be Fontaine's map.
	For $i\in \Z$, we denote by $\O_K\{i\}:=(\ker \theta)^i/(\ker\theta)^{i+1}$ the $i$-th Breuil--Kisin--Fargues twist of $\O_K$. This is a finite free $\O_K$-module of rank $1$, and any choice of a generator of the principal ideal $\ker \theta$ induces an isomorphism $\O_K\{i\}\cong\O_K$. By \cite[Example~4.24]{BMS}, if $K$ contains all $p$-power unit roots, then $\O_K\{i\}=\O_K(i)$ can be identified with the $i$-th Tate twist.
	For any $\O_K$-module $M$ we set
	\[ M\{i\}:=M\otimes_{\O_K}\O_K\{i\}.\]
	Let us reiterate that this is isomorphic to $M$, but only non-canonically so.
\end{Definition}
\begin{Theorem}\label{t:v-Brauer}
	Let $K$ be any perfectoid field. 
	\begin{enumerate}
	\item Let $X$ be a perfectoid space over $K$, or assume that $\Char K=0$ and let $X$ be a smooth rigid space over $K$. Then
	\[ R\nu_{\ast}\bOx_v=\bOx_{\et}.\]
	\item If $X$ is perfectoid, then also
	\[R\nu_{\ast}\O^\times_v=\O^\times_{\et}.\]
	\item If $X$ is rigid, we instead have for  $i\geq 1$ a natural ``Hodge--Tate logarithm'' isomorphism
	\[\HT\log :R^i\nu_{\ast}\O^\times_v\isomarrow \Omega^i_{X_{\et}}\{-i\}.\]
	\end{enumerate}
\end{Theorem}
The last part generalises \cite[Proposition~2.21, Corollary~2.28]{heuer-v_lb_rigid}. In particular, we obtain the following generalisation of Theorem~2.29 in \textit{loc.\ cit.}:
\begin{Corollary}\label{c:mult-HT}
	 There is a ``multiplicative Hodge--Tate'' spectral sequence
\[ E_2^{ij}:=\begin{dcases}\begin{rcases}
H^i_{\et}(X,\O^\times)& \text{ if }j=0\\
H^i_{\et}(X,\Omega_X^j\{-j\})& \text{ if }j>0	
\end{rcases}\end{dcases}\Rightarrow H^{i+j}_{v}(X,\O^\times).
\]
\end{Corollary}
\begin{Remark}
	This is in analogy with the Hodge--Tate spectral sequence  \cite[\S3.3]{Scholze2012Survey}
		\[ H^i_{\et}(X,\Omega_X^j\{-j\})\Rightarrow H^{i+j}_{v}(X,\O).
		\]
	If $X$ is proper and $K$ is algebraically closed, then $H^{i+j}_{v}(X,\O)=H^{i+j}_{\et}(X,\Q_p)\otimes_{\Q_p}K$  and the latter sequence degenerates \cite[Theorem~13.3.(ii)]{BMS}. It is natural to ask whether a similar result holds for the multiplicative Hodge--Tate sequence, in accordance with   \cite[Theorem~1.3.2a]{heuer-v_lb_rigid}. It will be shown in \cite{Gerth2024HT-rigid-coeff} that this is indeed the case.

	We note that these two cases of $\G_a$ and $\G_m$ combine to show that Corollary~\ref{c:mult-HT} holds more generally for $\O^\times$ replaced with $G(\O)$ where $G$ is any commutative linear algebraic group over an algebraically closed $K$.
\end{Remark}
\begin{Corollary}
	Let $X$ be an affinoid perfectoid space over $K$, then for any $j\geq 1$ we have
	\[H^j_{\et}(X,\O^\times)=H^j_{\et}(X,\bOx).\]
	In particular, $H^j_{\et}(X,\O^\times)$ is uniquely $p$-divisible.
\end{Corollary}
\begin{Remark}
\begin{enumerate}
	\item That $\Pic(X)=H^1_{\et}(X,\O^\times)$ is uniquely $p$-divisible was already observed by Bhatt--Scholze in \cite[Corollary 9.7]{MR4502597}. The case of $j\geq 2$ is related to \cite[Theorem~4.10]{Cesnavicius_Brauergroup}, which proves that $H^j_{\et}(\Spec(R),\O^\times)$ is uniquely $p$-divisible for $j\geq 2$ where $R=\O(X)$. 
\item We will see in \S6 that the assumption on $X$ to be affinoid is necessary.
\end{enumerate}
\end{Remark}
\begin{proof}
	It follows from Theorem~\ref{t:v-Brauer}.1-2 that $R\nu_{\ast}\O^\times_{1}=\O^\times_1$. The first statement then follows from Lemma~\ref{l:1+mO-acyclic-on-aff-perf}. The last part follows from Lemma~\ref{l:bOx-p-divisible}.  
\end{proof}
\begin{proof}[Proof of Theorem~\ref{t:v-Brauer}]
	We begin with the first part: For this we divide $\nu$ into 
	\[
	\nu:X_v\xrightarrow{w} X_{\proet}\xrightarrow{u} X_{\et}.\]
	We have already seen in Corollary~\ref{c:bOx-for-rigid-proet} that the statement holds for $u$, it thus suffices to prove that $Rw_{\ast}\bOx_v=\bOx_{\proet}$. As for rigid $X$ the site $X_{\proet}$ is locally perfectoid, it suffices to consider the case of affinoid perfectoid $X$. By \cite[Lemma 7.18]{etale-cohomology-of-diamonds}, there is then a pro-\'etale cover $X'\to X$ by a strictly totally disconnected space, so we can further reduce to this case. Here we crucially use that for perfectoid spaces, we are working with the pro-\'etale site introduced in \cite[Definition~8.1]{etale-cohomology-of-diamonds}.

	As $\bOx_v(Y)=\bOx_v(Y^\flat)$ for any perfectoid space $Y$ over $K$, we can reduce to characteristic $p$.
	
	By Lemma~\ref{l:bOx-in-arbitrary-char}, we have $w_{\ast}\bOx_v=\bOx_{\proet}$. It therefore suffices to prove that for $j\geq 1$,
	\[H^j_v(X,\bOx)=1.\]
	For this we can now argue as in \cite[Proposition 14.7]{etale-cohomology-of-diamonds}:
	Let $\alpha\in H^j_v(X,\bOx)$.
	 By locality of cohomology, there is an affinoid perfectoid $v$-cover $f:Y\to X$ that trivialises $\alpha$. Considering the \cH-to-sheaf sequence of $f$, we see that after refining $f$ we can assume that $\alpha$ is in the image of the map
	\[ \cH^q(Y\to X,\bOx)\to H^j_v(X,\bOx).\]
	Arguing as in \cite[Proposition~3.17]{heuer-diamantine-Picard}, we can write $Y\to X$ as the tilde-limit $Y\approx \varprojlim_{i\in I} Y_i$ of a cofiltered inverse system of open subspaces $Y_i\subseteq \B^{n}_X\to X$ of closed perfectoid balls relative over $X$.  Then also $Y_{/X}^{\times n}:=Y\times_X\dots\times_XY$ ($n$ copies) is affinoid perfectoid, and by Proposition~\ref{p:cohom-comparison}, we have $\bOx(Y_{/X}^{\times n})=\varinjlim_{i\in I} \bOx(Y_{i/X}^{\times n})$. It follows that
	\[\textstyle \cH^q(Y\to X,\bOx)=\varinjlim_{i\in I} \cH^q(Y_i\to X,\bOx),\]
	so $\alpha$ is already trivialised by one of the $Y_i\to X$. But by Lemma~\ref{l:local-splitting-over-stds} below, this map has a section, thus $H^j_v(X,\bOx)\to H^j_v(Y_i,\bOx)$ is injective and $\alpha=1$. This proves the first part.
	
	The second part follows from the first using that $R\nu_{\ast}\O^\times_{1v}=\O^\times_{1\et}$ by Lemma~\ref{l:1+mO-acyclic-on-aff-perf}.
	
	For the third part,	it follows from the exponential sequence that for $i\geq 1$ the map
	\[ \exp:R^i\nu_{\ast} \O\isomarrow  R^i\nu_{\ast} \O^\times\tf\]
	is an isomorphism. Here we again use Lemma~\ref{l:bOx-in-arbitrary-char} for injectivity in the case of $i=1$. Since $R\nu_{\ast}\mu_{p^\infty}=\mu_{p^\infty}$, the Kummer sequence shows that $R^i\nu_{\ast} \O^\times=R^i\nu_{\ast} \O^\times\tf$ for $i\geq 1$.  Finally, by \cite[Proposition~3.23]{Scholze2012Survey} and \cite[Proposition~2.25]{heuer-v_lb_rigid}, we have $R^i\nu_{\ast} \O=\Omega_X^i\{-i\}$.
\end{proof}
The following lemma was used in the proof:
\begin{Lemma}[{\cite[Lemma~9.5]{etale-cohomology-of-diamonds}}]\label{l:local-splitting-over-stds}
	Let $X$ be a strictly totally disconnected space and let $f:V\subseteq \B^n_X\to X$ be an open subspace of a closed perfectoid ball relatively over $X$ such that $f$ is surjective. Then $f$ admits a splitting $X\to V$.
\end{Lemma}
\begin{proof}
	We can reduce to characteristic $p$, where the statement is equivalent when we replace $\B^n_X$ by the unperfected ball $X\times_K\Spa(K\langle X_1,\dots,X_n\rangle)$.
	The cited lemma therefore proves the desired statement in the case that $X=\Spa(C,C^+)$ is connected. It is already used in \cite{etale-cohomology-of-diamonds} that this implies the general statement, which can be seen as follows: Write $X=\Spa(R,R^+)$ and let $Z=\Spa(C,C^+)\hookrightarrow X$ be any connected component, then the map $R^+\to C^+$ is surjective: This can be seen by writing $Z$ as the intersection of all split open immersions $U\subseteq X$ containing $Z$, as for these $\O^+(X)\to \O^+(U)$ is surjective.
	
	We know that $f$ has a section over $Z$, corresponding to a map $R^+\langle X_1,\dots,X_n\rangle\to C^+$. Any choice of lifts of the images of $X_1,\dots,X_n$ to $R^+$ extends $Z\to \B_X^n$ to a section $X\to \B_{X}^n$. Let $V_Z\subseteq X$ be the intersection of $X$ and $V$ inside $\B_{X}^n$, then $f$ is split over $V_Z$. The $V_Z$ for all components $Z$ now form an open cover of $X$. Since $X$ is strictly totally disconnected, this cover has a splitting. Composing the splittings gives the desired map.
\end{proof}

\section{Picard groups of perfectoid covers}
As before, let $K$ be a perfectoid field over $\Z_p$.  Let $C$ be the completion of an algebraic closure of $K$. If $K$ is of characteristic $p$, we fix an untilt $K^\sharp$ over $\Q_p$. Then tilting induces an equivalence $\Perf_K=\Perf_{K^\sharp}$ and we denote by $\O^\sharp$ the structure sheaf on $\Perf_{K^\sharp}$.

Our first main result on Picard groups of perfectoid covers is the following general theorem. It may look slightly technical in its formulation, but all of our following more concrete results will be deduced from this:

\begin{Theorem}\label{t:Pic(wt B)-vs-Pic-special-fibre}
	Let $(\mathfrak X_i)_{i\in I}$ be a cofiltered inverse system of smooth qcqs formal $\O_K$-schemes with affine transition maps. Assume that the adic generic fibre $X_\infty$ of $\varprojlim_{i\in I} \mathfrak X_i$ is perfectoid. Let $X_i$ be the generic fibre and $\overbar{X}_i$ the special fibre of $\mathfrak X_i$.  Moreover:
	\begin{enumerate}[label=(\roman*)]
		\item If $\Char K=0$, assume that $H^j(X_\infty,\O)=0$ for $j=1,2$.
		\item If $\Char K=p$, assume  $H^j_{\proet}(X_{\infty},\U)=0$ for $j=1,2$. For example, this holds if $H^j(X_{\infty,C},\O^\sharp)=0=H^j(X_{\infty,C},\Q_p)$ and either $K=C$  or $\O^+(X_{\infty,C})\aeq \O_C$.
	\end{enumerate}
		Then the pullback map $\Pic(X_i)\to \Pic(X_\infty)$ induces a natural  isomorphism
	\[\textstyle \varinjlim_{i\in I} \Pic(\overbar{X}_i)[\tfrac{1}{p}]\isomarrow\Pic(X_\infty)[\tfrac{1}{p}].\]
	If (i) or (ii) only hold for $j=1$, it is still true that the kernel of $\varinjlim\Pic(X_i)\tf\to \Pic(X_\infty)\tf$ equals that of the reduction map $\varinjlim\Pic(X_i)\tf\to \varinjlim\Pic(\overbar X_i)\tf$.
\end{Theorem}
One setup in which the two conditions are satisfied is if the $\mathfrak X_i$ are all affine and $X_\infty$ is affinoid perfectoid, due to \Cref{l:1+mO-acyclic-on-aff-perf}. We then essentially recover Proposition~\ref{p:Gabber-Ramero-5.4.42}.

More interestingly, as a second example of this setup, we can answer our first new instances of Question~\ref{q:Pic(X)->Pic(X^perf)-kernel}: If $\Char K=p$, let us denote by $F$ the Frobenius morphism on $X$.

\begin{Corollary}\label{c:Pic(X^perf)-for-X-with-trivial-cohom}
	Let $ X$ be a geometrically connected smooth proper rigid space over a perfectoid field $K$ of characteristic $p$ with good reduction $\overbar{X}$. If $F$ acts nilpotently on $H^j(X,\O)$ for $j=1,2$, then
	\[\Pic( X^\perf)=\Pic(\overbar X)\tf.\]
	Equivalently, this happens if $H^j(X,\Z_p)=0$ for $j=1,2$. 
\end{Corollary}
\begin{Example}
	\begin{enumerate}
\item For $ X=\P^n$, this recovers Dorfsman-Hopkins' results \cite{dorfsman2019projective} that
\[\Pic(\P^{n,\perf})= \textstyle\varinjlim_F \Pic(\P^{n})=\Z\tf.\]
\item More generally, the Corollary applies to smooth proper toric varieties, here $\O$ is acyclic by Demazure vanishing \cite[Corollary~7.4]{DanilovToric}. So we recover \cite[Theorem~1.1]{perfectoid_covers_toric}.
\item To give a new example, let $X$ be a supersingular abelian variety $K$ with good reduction $\overbar{X}$. Then Frobenius acts nilpotently on $H^{j}(X,\O)$ for $j\geq 1$ and the corollary applies.
\item The same works for other kinds of ``supersingular'' varieties, e.g.\ for supersingular K3 surfaces with good reduction, etc.
\end{enumerate}
\end{Example}

\begin{proof}
	Let $\mathfrak X$ be the smooth formal model and set  $\mathfrak X_i:=\mathfrak X$ with transition maps $\mathfrak X_{i+1}\to \mathfrak X_i$ the absolute Frobenius. We claim that the conditions of Theorem~\ref{t:Pic(wt B)-vs-Pic-special-fibre} are satisfied.

	By \cite[Corollary~3.8]{PCT-char-p}, it follows from the supersingularity assumption that for $j=1,2$,
	\[H^j(X^\perf_C,\O^\sharp)=0 \quad \text{and} \quad H^j(X^\perf_C,\Q_p)=0.\]
	We also have $\O^+(X^\perf_C)=\O_C$.
	Thus condition (ii) is satisfied, and the theorem applies.
\end{proof}

\begin{Remark}
	Regarding Question~\ref{q:Pic(X)->Pic(X_infty)-surj},
	the last part of  Theorem~\ref{t:Pic(wt B)-vs-Pic-special-fibre} means that line bundles which are ``$p$-adically close'' on $X$ (in the sense that they agree mod $\mathfrak m$) become isomorphic on $X_\infty$. 
	If $K$ is an algebraically closed field of characteristic $0$, this is closely related to  \cite[Theorem~3.6.1]{heuer-geometric-Simpson-Pic}, which says that a similar phenomenon occurs for the universal pro-finite-\'etale cover $\wt X\to X$.
	In this regard, Theorem~\ref{t:Pic(wt B)-vs-Pic-special-fibre} gives an explanation from a different perspective that is more restrictive in its good reduction assumption on $X$, but more elementary (e.g.\ it doesn't use diamonds in any essential way), and applies in greater generality, including cases of characteristic $p$ and over non-algebraically closed base fields. 
	
	More importantly, Theorem~\ref{t:Pic(wt B)-vs-Pic-special-fibre} describes $\Pic(X_\infty)$ itself, and thus can be used to answer the second part of Question~\ref{q:Pic(X)->Pic(X_infty)-surj}, which \cite[Theorem~3.6]{heuer-geometric-Simpson-Pic} does not tell us anything about.
\end{Remark}
\subsection{$\bOx$ in the case of good reduction}
We begin the proof of Theorem~\ref{t:Pic(wt B)-vs-Pic-special-fibre} with the following series of lemmas. For these we can again work over any non-archimedean field $K$ of residue characteristic $p$, without assuming that this is perfectoid. 
We start by recalling two useful lemmas from the literature:

\begin{Lemma}[{\cite[Proposition~3.4.1]{Lutkebohmert_RigidCurves}}]\label{l:reduced-fibre-implies-integrally-closed-in-generic-fibre}
	Let $\mathfrak X$ be a reduced formal scheme of topologically finite presentation over $\O_K$ with reduced special fibre. Then for the generic fibre $\eta:X\to \mathfrak X$,
	\[\eta_{\ast}\O^+_{ X}=\O_{\mathfrak X}.\]
\end{Lemma}

\begin{Lemma}[{\cite[Lemma~6.2.4]{Lutkebohmert_RigidCurves}}]\label{l:BL8-11}
	Let $\mathfrak X$ be an affine smooth formal scheme over $\O_K$ with rigid generic fibre $X$ and special fibre $\overbar{X}$. Then the following natural maps are isomorphisms:
	\[\Pic_{\an}(X)\isomfrom\Pic_{\Zar}(\mathfrak X)\isomarrow\Pic_{\Zar}(\overbar{X}).\]
\end{Lemma}
\begin{proof}
	That the second map is an isomorphism follows from the same lifting argument as in  Lemma~\ref{l:bOx-in-arbitrary-char}, using that $\O_{\mathfrak X}/p$ is acyclic on $\mathfrak X$.

	 The cited lemma proves that the first map is surjective.
	We note that  \cite[\S6]{Lutkebohmert_RigidCurves} has a running assumption that $K$ is algebraically closed, but this is not used for the proof of the relevant Lemmas~6.2.3-4. 
	
	It follows that any line bundle on $X$ is trivial locally on $\mathfrak X$. Let $\lambda:X_\an\to \mathfrak X_\Zar$ be the natural morphism of sites, then this implies that $\Pic_{\an}(X)=H^1_{\Zar}(\mathfrak X,\lambda_\ast \O_X^\times)$. To see that the map $\Pic_{\Zar}(\mathfrak X)\to \Pic_\an(X)$ is injective, it thus remains to prove that  the following map is injective:
	\[  H^1_{\Zar}(\mathfrak X, \O_{\mathfrak X}^\times)\to H^1_{\an}(\mathfrak X,\lambda_\ast \O_X^\times).\]
	This follows from the long exact sequence of $H^0(\mathfrak X,-)$ applied to \Cref{l:barO-computes-line-bundles-on-special-fibre}.1 below.
\end{proof}
\begin{Remark}
	Many variants of this lemma with different proofs can be found in the literature for different assumptions on $K$: The earliest is by Gerritzen \cite[\S4]{Gerritzen-ZerlegungPicard} for discrete $K$, one can then even relax conditions on $\mathfrak X$ as shown by Hartl--L\"utkebohmert \cite[Lemma 2.1]{HartlLutk}. For algebraically closed $K$, it is shown by Bosch--L\"utkebohmert \cite[Lemma~8.11]{BL-stable-reduction-II}. Heinrich--van der Put consider stable $K$ \cite[\S1]{H-vdP_Picard}.
	The case of perfectoid $K$, which are rarely stable, can also be deduced from the algebraically closed case by Galois descent.
\end{Remark}

For the rest of this subsection, let $\mathfrak X$ be any smooth formal scheme with generic fibre $X$ and special fibre $\overbar{X}$. We consider the morphism of sites
\[\lambda:X_{\an}\to {\mathfrak X}_{\Zar}=\overbar{X}_{\Zar}.\]
We remind the reader that we denote by $\Gamma$ the value group of $K$.
\begin{Lemma}\label{l:barO-computes-line-bundles-on-special-fibre} 
	\begin{enumerate}
		\item The supremum norm defines a homomorphism $\lambda_{\ast}\O^\times_{X}\xrightarrow{\|- \|_{\sup}}\underline \Gamma$ of sheaves of abelian groups on $\mathfrak X_{\Zar}$ that fits into an exact sequence of sheaves
		\[ 1\to \O_{\mathfrak X}^{\times}\to \lambda_{\ast}\O^\times_{X}\xrightarrow{\|- \|_{\sup}}\underline{\Gamma}\to 0.\]
		This stays exact after applying $H^0(U,-)$ for any affine open subscheme $U\subseteq \mathfrak X$.
		\item\label{i:ses-of-lambda_astbOx} It induces a natural exact sequence on $\mathfrak X_{\Zar}$ 
		\[ 	\begin{tikzcd}
			1 \arrow[r] & \O_{\overbar{X}}^\times\tf \arrow[r] & \lambda_{\ast}\bOx \arrow[r,"\|-\|_{\sup}"] & \underline{\Gamma}\tf \arrow[r] & 0.
		\end{tikzcd}\]
	\end{enumerate}
\end{Lemma}
\begin{proof}	
	For the first part, consider any affine open $U=\Spf(A)\subseteq \mathfrak X$. Then by Lemma~\ref{l:reduced-fibre-implies-integrally-closed-in-generic-fibre},
	\[\lambda_{\ast}\O_X^{+}(U)=A,\quad \lambda_{\ast}\O_X(U)=A[\tfrac{1}{\varpi}].\]
	Since $A$ is smooth, we may replace $U$ by a  connected component to assume that the special fibre $A/\m$ is integral, in particular reduced. It follows by \cite[\S6.4.3]{BGR}\cite[Proposition~3.4.1]{Lutkebohmert_RigidCurves} that there is a surjection
	$\alpha:\O_K\langle X_1,\dots,X_n\rangle\to A$
	such that $\|-\|_{\sup}=|-|_{\alpha}$ is the residue norm of $\alpha$. In particular,
	\[\{f\in A[\tfrac{1}{\varpi}]\mid \|f\|<1 \}=\m A,\]
	as well as $\|A\|=|A|_{\alpha}=|K|=\Gamma$.
	
	Let now $f,g\in A[\tfrac{1}{\varpi}]$. To prove $\|fg\|=\|f\|\|g\|$, we may without loss of generality assume that $f,g \neq 0$. After rescaling, we may assume that we have $\|f\|=1=\|g\|$, thus $f,g\in A\backslash \m A$. We need to see that $\|fg\|\geq \|f\|\|g\|=1$, since $\leq$ is clear. 
	
	Suppose that $\|fg\|<1$, then by the above $fg \in \m A$ and thus the image of $fg$ in $A/\m$ vanishes. Since $A/\m$ is integral, this implies $f\in \m A$ or $g \in \m A$, a contradiction.
	
	As a consequence, we see that we have a short exact sequence of abelian groups
	\[1\to A^{\times}\to A[\tfrac{1}{\varpi}]^{\times} \xrightarrow{\|-\|} \Gamma\to 1,\]
	which implies $A[\tfrac{1}{\varpi}]^{\times}=K^\times\cdot A^{\times}$.
	In particular, while this is clearly not true for any $f\in A[\tfrac{1}{\varpi}]$, if $f\in A[\tfrac{1}{\varpi}]^\times$ then the value of  $\|f\|_{\sup}$ is stable under any localisation on $U$. This shows that $\|-\|_{\sup}$ glues to a multiplicative map
	$\lambda_{\ast}\O^\times_{X}\to\underline{\Gamma}$ as desired.
	
	To deduce part 2, consider the commutative diagram of abelian sheaves on $\overbar{X}_{\Zar}$
	\begin{center}
		\begin{tikzcd}
			0 \arrow[r] & {\O^{\times}_{\mathfrak X}/(1+\m \O_{\mathfrak X}}) \arrow[r] \arrow[d] & \lambda_{\ast}\O_X^\times/(1+\m \O_{\mathfrak X}) \arrow[d] \arrow[r,"\|-\|_{\sup}"] & \underline{\Gamma} \arrow[d] \arrow[r] & 0 \\
			0 \arrow[r] & \O_{\overbar{X}}^\times \arrow[r] & \lambda_{\ast}\bOx \arrow[r,"\|-\|_{\sup}",dotted] & \underline{\Gamma}\tf \arrow[r] & 0.
		\end{tikzcd}
	\end{center}
	The first column is an isomorphism. The second column becomes an isomorphism after inverting $p$ by Lemmas~\ref{l:bOx-in-arbitrary-char} and \ref{l:reduced-fibre-implies-integrally-closed-in-generic-fibre}.  The top row is exact by the first part. Hence the bottom row is exact after inverting $p$.
\end{proof}

\begin{Lemma}\label{l:R1lambda_bOx}
	We have $R^1\lambda_{\ast}\bOx=0$.
\end{Lemma}
\begin{proof}
	Let $\mathfrak U\subseteq \mathfrak X$ be an affine open with reduction $\overbar U\subseteq \overbar{X}$ and generic fibre $U$.
	Then by using first
	Lemma~\ref{l:bOx-in-arbitrary-char}.3 and then Lemma~\ref{l:BL8-11} we see that $R^1\lambda_{\ast}\bOx$ is the sheafification of \[\overbar U\mapsto  H^1_{\an}(U,\bOx)=\Pic(U)\tf=\Pic(\overbar U)\tf,\]
	which clearly sheafifies to zero.
\end{proof}
We can now prove that $\bOx$ on $X_{\et}$ computes line bundles on the special fibre $\overbar{X}$:
\begin{Proposition}\label{p:bOx-computes-line-bundles-on-special-fibre}
	Let $\mathfrak X$ be a smooth formal scheme with generic fibre $X$ and special fibre $\overbar{X}$ with morphism of sites $\lambda:X_{\an}\to \overbar{X}_{\Zar}$.
	The map $\O^\times_{\overbar{X}}\to \lambda_{\ast}\bOx$ induces an isomorphism
	\[\Pic(\overbar{X})[\tfrac{1}{p}]\to  H^1_{\et}(X,\bOx).\]
\end{Proposition}
\begin{proof}
	We can assume that $\overbar{X}$ is connected.
	By Lemma~\ref{l:bOx-in-arbitrary-char}.3, we have $H^1_{\an}(X,\bOx)=H^1_{\et}(X,\bOx)$, so we can work in the analytic topology.
	By Lemma~\ref{l:R1lambda_bOx}, we have 
	\[H^1_{\Zar}(\overbar X,\lambda_{\ast}\bOx)= H^1_{\an}({X},\bOx).\]
	 The exact sequence of Lemma~\ref{l:barO-computes-line-bundles-on-special-fibre}.\ref{i:ses-of-lambda_astbOx} now induces an exact sequence
	\[\bOx(X)\to \Gamma\tf\to  H^1_{\Zar}(\overbar{X},\O_{\overbar{X}}^\times)[\tfrac{1}{p}]\to H^1_{\Zar}(\overbar{X},\lambda_{\ast}\bOx)\to H^1_{\Zar}(\overbar{X},\underline{\Gamma})\to 0.\]
	The first map is surjective since $K^\times/(1+\m)\to \Gamma$ is surjective and $\bOx$ is $p$-divisible.
	The last term vanishes since $\overbar{X}$ is smooth and connected, hence irreducible: This implies that any constant sheaf is Zariski-flabby (\cite[``Un example amusant'', before \S3.4]{Grothendieck_Tohoku}, see also \cite[{Tag 02UW}]{StacksProject}). Thus the second arrow is an isomorphism.
\end{proof}
\subsection{Proof of Main Theorem}
\begin{proof}[Proof of Theorem~\ref{t:Pic(wt B)-vs-Pic-special-fibre}]
	In the case of (i), we have by Proposition~\ref{p:bOx} an exact sequence
	\[H^1_{\et}(X_\infty,\O)\xrightarrow{\exp} H^1_{\et}(X_\infty,\O^{\times}[\tfrac{1}{p}])\to H^1_{\et}(X_\infty,\overbar\O^\times)\to H^2_{\et}(X_\infty,\O).\]
	Since the outer terms vanish by assumption, the middle map is an isomorphism. The second term is $\Pic(X_\infty)\tf$ by quasi-compactness of $X_\infty$. By Proposition~\ref{p:cohom-comparison}, we thus have
	\[\textstyle \Pic(X_\infty)\tf=  H^1_{\et}(X_\infty,\overbar\O^\times)= \varinjlim_{i\in I}H^1_{\et}(X_i,\overbar\O^\times).\]
	By Proposition~\ref{p:bOx-computes-line-bundles-on-special-fibre},
	\[ H^1_{\et}(X_i,\overbar\O^\times)=\Pic(\overbar{X}_i)[\tfrac{1}{p}].\]
	In the colimit over $i\in I$, this gives the desired isomorphism.
	
	The weaker statement when (i) only holds for $j=1$ follows from the same identifications.
	
	For condition (ii), i.e.\ in characteristic $p$, we instead use the sequence
	\[H^1_{\proet}(X_\infty,\O^\times_1)\rightarrow H^1_{\proet}(X_\infty,\O^{\times}[\tfrac{1}{p}])\to H^1_{\proet}(X_\infty,\overbar\O^\times)\to H^2_{\proet}(X_\infty,\O^\times_1)\]
	and then argue as for (i). To see that the second set of conditions in the theorem implies that the outer terms vanish, we first note that the pro-\'etale logarithm sequence
	\[ 0\to \Q_p\to \U\to \O^\sharp\to 0\]
	shows that the assumptions imply for $j=1,2$ that
	$H^j_{\proet}(X_{\infty,C},\U)=0$.
	Moreover, the last assumption implies $H^0(X_{\infty,C},\U)=1+\m\O_C$.
	Let $G=\Gal(C|K)$, then by the Cartan--Leray spectral sequence of $X_{\infty,C}\to X_{\infty,K}$ this implies that for $j=1,2$,
	\[ H^j_{\cts}(G,1+\m\O_C)=H^j(X_{\infty},\U).\]
	But the left-hand-side also computes $H^j_{\proet}(\Spa(K),\U)$, which vanishes for $j\geq 1$ by \Cref{l:1+mO-acyclic-on-aff-perf}.
\end{proof}

\section{The case of abeloid varieties}
As our main application of Theorem~\ref{t:Pic(wt B)-vs-Pic-special-fibre}, we now compute an interesting example in which the morphism from Question~\ref{q:Pic(X)->Pic(X^perf)-kernel} is far from being an isomorphism, namely for universal covers of abeloid varieties. This is of particular relevance to Question~\ref{q:Pic(X)->Pic(X^perf)-kernel} as via the theory of Albanese varieties, it has implications for  proper varieties in general. Independently, we apply our description in \cite{heuer-isoclasses} to obtain a pro-\'etale uniformisation result for abeloid varieties.

Our setup throughout this section is as follows: Let $K$ be any perfectoid field over $\Z_p$, and let $A$ be an abeloid variety over $K$. An abeloid is by definition a connected smooth proper rigid group, we refer to \cite{Lutkebohmert_abeloids} for background.  We consider the ``$p$-adic universal cover''
\[ \wt A=\varprojlim_{[p]} A,\]
studied in \cite{perfectoid-covers-Arizona}.
If $A$ has good reduction, it is easy to see that this is a perfectoid space \cite[Lemme~A16]{pilloni2016cohomologie}. The first result of this section is the following description of its Picard group:
\begin{Theorem}\label{t:Pic(wt B)}
	Let $B$ be an abeloid variety  over $K$  with good reduction $\overbar{B}$ over $k$. Then 
	\[ \Pic(\wt B)=\Pic(\overbar{B})[\tfrac{1}{p}]\]
\end{Theorem}

For the proof we would like to apply Theorem~\ref{t:Pic(wt B)-vs-Pic-special-fibre}. The goal of this section is to verify that the conditions for this are satisfied.

Let more generally $A$ be any abeloid variety over $K$. Then $\wt A=\varprojlim_{[p]} A$ is perfectoid if $\Char K=p$, or if $K$ is algebraically closed \cite[Theorem~1]{perfectoid-covers-Arizona}, or if $A$ has good reduction \cite[Lemme~A16]{pilloni2016cohomologie}. While we expect it to be perfectoid in general, we can for the purpose of this section simply consider $\wt A$ as a spatial diamond outside of these cases.
The same applies to
\[ \wtt A:=\varprojlim_{[N],N\in \N} A.\]

The term ``universal cover'' might be slightly confusing in this setting: If $K$ is algebraically closed, then $\wtt A$ is the universal pro-finite-\'etale cover of $A$ in the sense of \cite[\S4.3]{heuer-v_lb_rigid}, i.e.\ it has a universal mapping property for pro-finite-\'etale maps. At least in the case of good reduction, the space $\wt A$ is the universal cover of $A$ in the sense of \cite[Definition 13.1]{ScholzePSandApplications}, where it originated from the universal covers of $p$-divisible groups of Scholze--Weinstein.

The reason why we call $\wt A$ the $p$-adic universal cover is the following topological universal property, which we think of as a $p$-adic analogue of saying that $\wt A$ is simply connected:
\begin{Proposition}\label{p:H^i(wtA,O^+)}
	Let $A$ be an abeloid variety over $K$. Then for any $i\in \Z_{\geq 0}$,
	\[H^i_{v}(\wt A,\O^+)\aeq \begin{cases}\O_K&\quad \text{ for }i=0,\\0& \quad \text{ for }i>0.\end{cases}\]
	Second, the natural map $H^i_v(K,\Zp)\to H^i_v(\wt A,\Zp)$ is an isomorphism. In particular, if $K$ is algebraically closed,
	\[H^i_v(\wt A,\Zp)=\begin{cases}\Z_p&\quad \text{ for }i=0,\\0& \quad \text{ for }i>0.\end{cases}\]
	The same applies to $\wtt A$, where $H^i_v(K,\Z_l)\isomarrow H^i_v(\wtt A,\Z_l)$ for any prime $l$.
\end{Proposition}
\begin{Remark}
This is mentioned in \cite[Bhatt's lecture, Proposition 2.2.1]{Lectures_Arizona} in the case of good reduction over an algebraically closed field of characteristic $0$.
\end{Remark}
\begin{proof}
	We first treat the case that $K$ is algebraically closed.
	
In the case that $\Char K=0$, we recall the following facts from the classical rigid theory: 
	\begin{Lemma}\label{l:et-cohom-of-abeloids}
	If $\Char(K)=0$, then for any $N\in \N$, we have
	\[H^{i}_{\et}(A,\Z/N\Z)=\wedge^i A[N](K)^\vee,\]
	and the map $[p]:A\to A$ induces on this the map defined by
	$\cdot p:A[N](K)^\vee\to A[N](K)^\vee$.
\end{Lemma}  
\begin{proof}
	Since we do not know a reference for this fact for abeloids, we briefly explain how to see this: By the Kummer sequence, we have $H^1(A,\mu_{N})=\Pic(A)[N]=A^\vee[N](K)$.  The Weil pairing thus induces an isomorphism $H^1(A,\Z/N\Z)=A[N](K)^\vee$. One then deduces like for abelian varieties that $H^{\ast}_{\et}(A,\Z/N\Z)=\wedge^\ast H^{1}_{\et}(A,{\Z/N\Z})$ using the rigid analytic K\"unneth formula (see e.g.\ \cite[Corollary 5.7]{heuer-relative-HT}).
	\end{proof}
	Since $\wt A$ is spatial, we have by \cite[Proposition 14.8]{etale-cohomology-of-diamonds} a comparison isomorphism
	\[H^i_{v}(\wt A,{\Z/p^n\Z})=H^i_{\et}(\wt A,{\Z/p^n\Z}).\]
	Using \cite[Proposition 14.9]{etale-cohomology-of-diamonds} and Lemma~\ref{l:et-cohom-of-abeloids}, we further have
	\[ H^i_{\et}(\wt A,{\Z/p^n\Z})=\varinjlim_{[p]^{\ast}}H^i_{\et}(A,{\Z/p^n\Z})=\varinjlim_{p^i}\wedge^i A[p^n](K)^\vee= \begin{cases}\Z/p^n\Z\quad &\text{ for }i=0,\\0 \quad &\text{ for }i>0.\end{cases}\]

	We now take the limit over $n$.  By Corollary~\ref{c:replete-derived-complete}, we have $\Rlim_n{\Z/p^n\Z}=\Zp$ and thus
	\[ \mathrm R\Gamma_{v}(\wt A,\Zp)=\mathrm \RGamma_{v}(\wt A,\Rlim \Z/p^n\Z)=\Rlim \mathrm \RGamma_{v}(\wt A, \Z/p^n\Z)=\Z_p. \]
	This finishes the proof of part 1.
	
	For the case of characteristic $p$, we could deduce the result from the fact that one can find an abeloid $A'$ over $K^\sharp$ such that $\wt A$ untilts to $\wt A'$ (see  \cite[\S5]{wear2020perfectoid}), which identifies \'etale sites. Alternatively, we now give a direct argument, which we feel is worth recording: It suffices to consider the case $n=1$ where $\Z/p^n\Z=\F_p$. We will show by induction on $k$ that $H^k_\et(\wt A,\F_p)=0$ for $k>1$. Suppose we know this already in degrees $<k$.
	In the colimit $\varinjlim_{[p]^\ast}$, the K\"unneth formula for $A$ then induces a decomposition
	\[ H^k(\wt A\times \wt A,\F_p)=\oplus_{i+j=k} H^i(\wt A,\F_p)\otimes H^j(\wt A,\F_p)=H^k(\wt A,\F_p)\oplus H^k(\wt A,\F_p)\]
	where we have used that all summands with $0<j<k$ vanish by induction assumption. With respect to this decomposition, we see that the multiplication $\wt A\times \wt A\to \wt A$ induces on cohomology the map
	\[ H^n(\wt A,\F_p)\to H^n(\wt A\times \wt A,\F_p),\quad x\mapsto (x,x)\]
	because we can detect the contribution to each factor by precomposing with the respective inclusions $\wt A\to \wt A\times \wt A$.
	Considering the $p$-fold multiplication $(\wt A)^p\to \wt A$ and composing with the diagonal embedding $\wt A\to (\wt A)^{p}$, we deduce that
	\[ [p]^\ast: H^k(\wt A,\F_p)\to  H^k(\wt A,\F_p)\]
	is given by the morphism $\cdot p$, which is clearly $=0$. But on the other hand, $[p]$ is an isomorphism on $\wt A$, so $[p]$ has to be an automorphism. This is only possible if $H^k(\wt A,\F_p)=0$.
	
	For part 2, it suffices to prove this for $\O^+$ replaced by $\O^+/\varpi$, as by Corollary~\ref{c:replete-derived-complete}
	\[\RGamma_{v}(\wt A,\O^+)\aeq \textstyle \Rlim_n \RGamma_{v}(\wt A,\O^+/\varpi^n).\]
	Using \cite[Propositions 8.5, 8.8 and 14.9]{etale-cohomology-of-diamonds}, we see that
	\[H^i_{v}(\wt A,\O^+/\varpi^n)=H^i_{\et}(\wt A,\O^+/\varpi^n)=\textstyle\varinjlim_{[p]^\ast} H^i_{\et}(A,\O^+/\varpi^n).\]
	If $\Char(K)=0$, the Primitive Comparison Theorem \cite[Theorem~1.3]{Scholze_p-adicHodgeForRigid} shows that for any $i\in \Z_{\geq 0}$ we have
	\[\varinjlim_{[p]^\ast} H^i_{\et}(A,\O^+/p^n)\aeq \varinjlim_{[p]^\ast} H^i_{\et}(A,\Z/p^n\Z)\otimes_{\Z_p}\O_K\aeq \begin{cases}\O_K/p^n&\quad \text{ for }i=0,\\0& \quad \text{ for }i>0.\end{cases}\]
	by the first part.
	This shows the statement for $H^i_{v}(\wt A,\O^+)$ when $\Char(K)=0$.
	
	 The case of $\Char(K)=p$ follows by the same argument from the characteristic $p$ version of  the Primitive Comparison Theorem \cite[Theorem~1.1]{PCT-char-p}.

	To deduce the general case, let $C$ be the completion of an algebraic closure of $K$ and consider the base change $A_C$ of $A$. Then we have a pro-finite-\'etale $G:=\Gal(C|K)$-torsor
	\[ \wt A_C\to \wt A_K.\]
	By the Cartan--Leray spectral sequence, we thus have
	\[ H^i_{\cts}(G,\O_K)\aeq H^i_v(\wt A_K,\O^+).\]
	We claim that the group on the left vanishes for $i\geq 1$ (compare \cite[Proposition~10]{tate1967p}). Indeed, by comparing to the $G$-torsor $\Spa(C)\to \Spa(K)$, we deduce that
	\[ H^i_{\cts}(G,\O_K)=H^i_v(\Spa(K),\O^+),\]
	and the right hand side is $\aeq 0$ as $\Spa(K)$ is affinoid perfectoid. 
\end{proof}
Before going on, we note as an application a more general form of \Cref{l:et-cohom-of-abeloids} which includes the case of characteristic $p$ and which seems useful to record in the literature:
\begin{Corollary}
	Let $A$ be any abeloid variety over any algebraically closed non-archimedean field. Let $N\in \N$ (we allow $\Char K$ to divide $N$). Then for any $i\in \N$,
	\[H^{i}_{\et}(A,\Z/N\Z)=\wedge^i H^{1}_{\et}(A,\Z/N\Z),\]
\end{Corollary}
\begin{proof}
	By \Cref{l:et-cohom-of-abeloids}, we may assume that $K$ has residue characteristic $p$.
	We consider the Cartan--Leray sequence of the cover $\wt A\to A$. By Proposition~\ref{p:H^i(wtA,O^+)}, this is of the form \[R\Gamma_\et(A,\Z/N\Z)=R\Gamma_\cts(T_pA,R\Gamma_\et(\wt A,\Z/N\Z))=R\Gamma_{\cts}(T_pA,\Z/N\Z)\]
	Since $T_pA\cong \Z_p^{2g}$ is a free $\Z_p$-module with trivial action on $\Z/N\Z$, this is computed by a Koszul complex, from which we obtain the desired result.
\end{proof}

Proposition~\ref{p:H^i(wtA,O^+)} shows that conditions (i) or (ii) of Theorem~\ref{t:Pic(wt B)-vs-Pic-special-fibre} are satisfied. To understand  the output of  Theorem~\ref{t:Pic(wt B)-vs-Pic-special-fibre} in this case, we also need the following statements:
\begin{Corollary}\label{c:Pic(wt A)-is-p-divisible}
	We have $\Pic(\wt A)=\Pic(\wt A)\tf$. Similarly, $\Pic(\wtt A)=\Pic(\wt A)\otimes \Q$.
\end{Corollary}
\begin{proof}
	For $\wt A$, this is clear in characteristic $p$ since $\O^\times$ is then $p$-divisible as $\wt A$ is perfectoid. 
	If $\Char K=0$, consider the comparison of Kummer exact sequences along $\wt B\to \Spa(K)$:
	\[\begin{tikzcd}[column sep =0.4cm]
\dots\arrow[r]&	H^1(\wt B,\mu_p)\arrow[r]& H^1(\wt B,\O^\times)\arrow[r,"{[p]}"]& H^1(\wt B,\O^\times)\arrow[r] &	H^2(\wt B,\mu_p)\arrow[r]&\dots\\
	\dots \arrow[r]& H^1(K,\mu_p)\arrow[r]\arrow[u,"\sim"labelrotatep]& 1\arrow[r]\arrow[u]& 1\arrow[r]\arrow[u]&  H^2(K,\mu_p)\arrow[u,"\sim"labelrotatep]\arrow[r]&\dots
	\end{tikzcd}\]
	The middle terms in the bottom row vanish since $K$ is perfectoid and thus $\Pic_v(\Spa(K))=1$. The first and last vertical arrows are isomorphisms by Proposition~\ref{p:H^i(wtA,O^+)}. It follows that the middle morphism in the top row is an isomorphism. The case of $\wtt A$ is analogous.
\end{proof}
\begin{Lemma}\label{l:colim[p]=colim p* on Pics}
	Let $A$ be an (algebraic) abelian variety over any field $L$. Then
	\[\textstyle\varinjlim_{[p]^{\ast}}\Pic(A)=\Pic(A)[\tfrac{1}{p}].\]
\end{Lemma}
\begin{proof}
	This is entirely classical:
	The statement is clear on $\Pic^0$, since $[p]^{\ast}L=L^{\otimes p}$ for any $L\in \Pic^0(X)$. It thus suffices to see that for the N\'eron--Severi group,
	\[\textstyle\varinjlim_{[p]^{\ast}}\NS(A)=\NS(A)\tf.\]
	This holds as a consequence of the Theorem of the Cube, see \cite[II.8.(iv)]{MumfordAV}
\end{proof}

\begin{proof}[Proof of Theorem~\ref{t:Pic(wt B)}]	
	Let $\mathfrak B$ be the abelian formal model of $B$.
	We apply Theorem~\ref{t:Pic(wt B)-vs-Pic-special-fibre} to
	\[\cdots\xrightarrow{[p]}\mathfrak B \xrightarrow{[p]}\mathfrak B.\]
	The condition (i) that $H^j(\wt B,\O)=0$ for $j=1,2$ holds by Proposition~\ref{p:H^i(wtA,O^+)}. Similarly, if $\Char K=p$, condition (ii) asks that $H^j(\wt B_C,\O^\sharp)=0$ and $ H^j(\wt B_C,\Q_p)=0$ for $j=1,2$ as well $\O^+(\wt B_C)\aeq \O_C$, which is satisfied by Proposition~\ref{p:H^i(wtA,O^+)}. The conclusion of the theorem is
	\[\Pic(\wt B)\tf= \textstyle\varinjlim_{[p]^{\ast}}\Pic(\overbar{B})[\tfrac{1}{p}].\]
	By Corollary~\ref{c:Pic(wt A)-is-p-divisible}, the left hand side equals $\Pic(\wt B)$.
	By Lemma~\ref{l:colim[p]=colim p* on Pics} applied with $L=k$, we see that the right hand side equals $\Pic(\overbar{B})[\tfrac{1}{p}]$. 
\end{proof}
\begin{Remark}
	The exact same arguments show that $\Pic(\wtt B)=\Pic(\overbar B)\otimes \Q$.
\end{Remark}
Using Theorem~\ref{t:Pic(wt B)}, we get a full answer to Question~\ref{q:Pic(X)->Pic(X_infty)-surj} for $\wt B=\varprojlim B$:
\begin{Definition}\label{d:tt}
	Let $G$ be a rigid group. Following \cite[Definition 2.6, Proposition 2.14]{heuer-geometric-Simpson-Pic}, we define the topological $p$-torsion subgroup $\wh{G}\subseteq G$ by $\wh{G}=\underline{\Hom}(\Z_p,G)$, this is a rigid open subgroup. By \cite[Proposition 2.14]{heuer-geometric-Simpson-Pic}, its $K$-points can be described by
	\[ \wh G(K)=\{x\in G(K)\mid [p^n](x)\xrightarrow{n\to \infty} 0\}.\]
	We refer to \cite[\S2]{heuer-geometric-Simpson-Pic} for a systematic treatment.
\end{Definition}
\begin{Corollary}\label{c:Q-in-case-of-B-of-good-reduction}
	Let $B^\vee$ be the dual abeloid variety. Then there is an exact sequence
	\[ 0\to \wh B^\vee(K) \to \Pic(B)\to \Pic(\wt B)\to \NS(\overbar{B})\tf/\NS(B)\to 0.\]
\end{Corollary}
\begin{proof}
	By \cite[Proposition 2.14.(4)]{heuer-geometric-Simpson-Pic}, the subset $\wh B^\vee(K)\subseteq B^\vee(K)=\Pic^0(B)$ consists precisely of those points that reduce to a $p$-torsion point in $B(k)$, therefore to $0$ in $B(k)\tf$.
\end{proof}

In future work, we will deduce from Theorem~\ref{t:Pic(wt B)} a more general version of this for abeloid varieties, via Raynaud uniformisation and $\Ext$-groups.

\section{Picard functors}

In this section, our goal is to show that in the proper case, Theorems~\ref{t:Pic(wt B)-vs-Pic-special-fibre} and \ref{t:Pic(wt B)} can be geometrised to a statement about Picard functors: 

To motivate the discussion, let us first consider  an abeloid variety $\pi:B\to \Spa(K)$ of good reduction over $K$ with its $p$-adic universal cover 
\[\wt \pi:\wt B\to \Spa(K).\]
For algebraically closed $K$ over $\Q_p$, we have in \cite{heuer-diamantine-Picard} studied the ``diamantine'' Picard functor of $\pi$ defined on $\Perf_{K,\et}$. One motivation was that we wanted to characterise all pro-finite-\'etale line bundles on $B$ by comparing to the  Picard functor of $\wt B$, defined as
\[ \uP_{\wt B}:=R^1\wt \pi_{\et\ast}\G_m:\Perf_{K,\et}\to \mathrm{Ab}.\]
It was left open whether $\uP_{\wt B}$ itself admits a nice description. Equipped with the preparations from the previous sections, we now show that this is indeed the case: Not only the Picard group but also the Picard functor can be described in terms of the special fibre. In fact, we can do this more generally in the setting of Theorem~\ref{t:Pic(wt B)-vs-Pic-special-fibre}. This gives some idea what to expect in general for ``Picard varieties'' of pro-proper objects e.g.\ in the pro-\'etale site.

\subsection{$v$-sheaves associated to schemes over the residue field}
We first need to introduce  $v$-sheaves associated to schemes over the residue field $k$:

\begin{Definition}
Let $X$ be any $k$-scheme. Then we denote by $X^\diamond$ the analytic sheafification of the presheaf $X^{\diamond\mathrm{pre}}$ on $\Perf_K$ defined by
\[ (S,S^+)\mapsto X(S^+/\mathfrak m),\]
where $S^+/\mathfrak m$ is considered as a $k$-algebra. Here we recall that $\mathfrak m\subseteq \O_K$ is the maximal ideal.
\end{Definition}
\begin{Lemma}\label{l:sheaf-assoc-to-special-fibre}
	$X^\diamond$ is already a $v$-sheaf. If $X$ is affine, then already $X^{\diamond\mathrm{pre}}$ is a $v$-sheaf.
\end{Lemma}
This is closely related to, but different from, the functor from perfect $k$-schemes to  $v$-sheaves defined in \cite[\S18]{ScholzeBerkeleyLectureNotes}. It is also closely related to Gleason's specialisation maps for diamonds, see \cite[\S4.5]{Gleason_Specialisation_for_Diamonds} for a detailed discussion of this relation.
\begin{proof}
	We first prove the statement about affine $X=\Spec(A)$. We choose any presentation
	\[  k[T_j|j\in J]\to k[T_i|i\in I]\xrightarrow{h} A\to 0\]
	where the images of the $T_j$ generate the kernel of $h$ as an ideal. This
	reduces us to the case that $X=\A^I$. Then $X^{\diamond\mathrm{pre}}$ is by definition the functor that sends $S^+$ to $(S^+/\mathfrak m)^{I}$, so it suffices to prove the result for $\A^1$. In this case, we have a short exact sequence on $\Perf_{K,v}$
	\[ 0\to \m\O^+\to \O^+\to \O^+/\mathfrak m\to 0 \]
	which evaluated on any $(S,S^+)$ induces a long exact sequence
	\[ 0\to \m S^+\to S^+\to \O^+/\mathfrak m(S,S^+)\to H^1_v(\Spa(S,S^+),\m\O^+)\]
	The last term vanishes: Indeed, $\O^+$ is almost acyclic on $\Perf_{K,v}$ by \cite[Proposition 8.8]{etale-cohomology-of-diamonds} and on the affinoid perfectoid space $Y=\Spa(S,S^+)$,  we have:
	\[ H^1_v(Y,\m\O^+)=\varinjlim_{n} H^1_v(Y,\varpi^{\epsilon_n}\O^+)=\varinjlim_{n} (H^1_v(Y,\O^+)\xrightarrow{\cdot \varpi^{\epsilon_n-\epsilon_{n+1}} }H^1_v(Y,\O^+))=0\]
	for any sequence $\epsilon_n\to 0$ in $\log|K|$. Thus $(\A^1)^{\diamond \mathrm{pre}}=\O^+/\mathfrak m$, which is a $v$-sheaf.
	
	To deduce the first part, we first note that for any affinoid perfectoid $T$ over $K$, any point $T\to X^\diamond$ induces a continuous map $|T|\to |X|$ sending a point $\Spa(C,C^+)\to T$ to the image of $\Spec(C^+/\m)\to X$ under the identification $|\Spa(C,C^+)|=|\Spec(C^+/\mathfrak m)|$.  
	
	Let now $S'\to S$ be a $v$-cover in $\Perf_{K}$ and suppose we are given a map $S'\to X^\diamond$ such that the maps $S'\times_SS'\rightrightarrows S'\to X^\diamond$ agree.  Arguing like in \cite[Corollary 8.6]{etale-cohomology-of-diamonds}, the fact that $S'\to S$ is a quotient map by \cite[Lemma~2.5]{etale-cohomology-of-diamonds} implies that the induced diagram of spectral spaces $|S'\times_SS'|\rightrightarrows |S'|\to |X|$ induces a continuous map $|S|\to |X|$. Thus the analytic sheafification in the definition of $X^\diamond$  reduces us to the affine case.
\end{proof}

We emphasize that $-^\diamond$ is just notation, and we think that $X^\diamond$ can only be a diamond if $X$ is $0$-dimensional. For example, one can show that $\A^{1\diamond}=\O^+/\mathfrak m$ is not a diamond.

\subsection{The Picard functor of $\wt B$}
We can now ``geometrise'' Theorem~\ref{t:Pic(wt B)-vs-Pic-special-fibre} as follows: Like in \cite{heuer-diamantine-Picard}, we consider for any perfectoid space $\pi:X\to \Spa(K)$ the Picard functor defined on perfectoid test objects:
\[\uP_X:=R^1\pi_{\et\ast}\G_m:\Perf_{K,\et}\to \mathrm{Ab}.\]
\begin{Theorem}\label{t:Pic-functor-for-wtB}
	Let $K$ be a complete algebraically closed field over $\Q_p$.
 As in Theorem~\ref{t:Pic(wt B)-vs-Pic-special-fibre}, let $(\mathfrak X_i)_{i\in I}$ be a cofiltered inverse system of smooth quasi-compact formal $\O_K$-schemes with affine transition maps and with perfectoid generic fibre $X_\infty$ of $\mathfrak X_\infty=\varprojlim \mathfrak X_i$. Assume moreover that the $\mathfrak X_i$ are all proper and that $H^j(X_\infty,\O^+)\aeq 0$ for $j=1,2$. Then
	\[ \uP_{X_\infty}\tf=\textstyle\varinjlim_{i\in I} (\uP_{\overbar X_i})^\diamond\tf,\]
where $\overbar{X}_i$ is the special fibre of $\mathfrak X_i$ and $\uP_{\overbar X_i}$ is its Picard variety over $k$. 
\end{Theorem}
We can thus give a geometric version of Theorem~\ref{t:Pic(wt B)} and Corollary~\ref{c:Q-in-case-of-B-of-good-reduction}:
\begin{Corollary}\label{c:Picard-functor-of-universal-cover}
	Let $B$ be an abeloid variety with good reduction over $K$. Let $\overbar{B}$ be the special fibre. Then there is a natural isomorphism of $v$-sheaves
	\[ \uP_{\wt B}=(\uP_{\overbar{B}})^\diamond\tf.\]
	In particular, there is a short exact sequence of abelian $v$-sheaves on $\Perf_K$:
	\[ 0\to \wh B^{\vee}\to \uP_B\to\uP_{\wt B}\to {\NS(\overbar B)\tf/\NS(B)}\to 0\]
	where $B^\vee$ is the dual abeloid, $\wh B^{\vee}$ is its topological $p$-torsion subgroup from \Cref{d:tt} and 
	 the last term is the locally constant $v$-sheaf of the abelian group $\NS(\overbar B)\tf/\NS(B)$.
\end{Corollary}
\begin{proof}
	We have already seen in the proof of Theorem~\ref{t:Pic(wt B)} that the conditions of Theorem~\ref{t:Pic-functor-for-wtB} are satisfied. This gives an isomorphism $\uP_{\wt B}\tf=(\uP_{\overbar{B}})^\diamond\tf$.
	Left-exactness of the sequence follows from \cite[Proposition 2.14]{heuer-geometric-Simpson-Pic}: This describe $\wh B^{\vee}$ as the subfunctor of $B$ of points that reduce to a torsion point in $\overline{B}$.  To see right-exactness, we first note that $B^\vee\to (\overbar{B}^\vee)^\diamond$ is surjective since $B^\vee$ is formally smooth. We now use:
	\begin{Lemma}\label{l:mupinfty-on-uni-cover}
	If $H^j(X_\infty,\mu_{p^\infty})=0$ for $j=1,2$, then $\uP_{X_\infty}=\uP_{X_\infty}\tf$.
\end{Lemma}
\begin{proof}
	Write $\pi:X_\infty\to \Spa(K)$. It suffices to prove  $R^j\pi_{\ast}\mu_{p^\infty}=0$ for $j=1,2$. For any affinoid perfectoid space $Y$, we have by \cite[Proposition~14.9]{etale-cohomology-of-diamonds}:
	\[H^j(X_\infty\times Y,\mu_{p^n})=\textstyle\varinjlim_{i\in I} H^j(X_i\times Y,\mu_{p^n}).\]
	The sheafification in $Y$ of the right hand side is by \cite[Corollary~4.8]{heuer-diamantine-Picard} given by the constant sheaf associated to
	\[ \varinjlim H^j(X_i,\mu_{p^n})=H^j(X_\infty,\mu_{p^n})=0.\qedhere\]
\end{proof}
Finally, $\varinjlim_{[p]^{\ast}} (\uP_{\overbar{B}})^\diamond\tf=(\uP_{\overbar{B}})^\diamond\tf$ by Lemma~\ref{l:colim[p]=colim p* on Pics}.
This shows Corollary~\ref{c:Picard-functor-of-universal-cover}.
\end{proof}

\subsection{Preparation: specialisation lemmas}
The proof of \Cref{t:Pic-functor-for-wtB} is technically demanding, and we need some technical lemmas as a preparations. We begin with specialisation lemmas that works in either characteristic. These will later be useful to study the boundary maps of the exponential sequence:

\begin{Lemma}\label{l:specialisation-for-bOx-in-products}
	Let $X$ and $Y$ be each a rigid or perfectoid space. Assume $Y$ is quasi-compact. 
	\begin{enumerate}
		\item The image of the following natural specialisation map 
		lies in the locally constant maps:
		\[ H^1_{\et}(X\times Y,\bOx)\to \Map(X(K),H^1_{\et}(Y,\bOx)).\]
		\item\label{i:commuting-bOx-with-profinite-sets} If $X$ is a locally profinite space, then this induces an isomorphism
		\[H^1_{\et}(X\times Y,\bOx)=\Map_{\lc}(X(K),H^1_{\et}(Y,\bOx)).\]
	\end{enumerate}
\end{Lemma}
\begin{proof}
	We can clearly reduce to the case that $X$ is affinoid. Let us at first also assume that $Y$ is affinoid.
	By rigid approximation, e.g.\ \cite[Proposition~3.17]{heuer-diamantine-Picard}, we can then further reduce to the case that $X$ and $Y$ are rigid spaces.
	In this case, we have \[H^1_{\et}(X\times Y,\O^\times)\tf=H^1_{\et}(X\times Y,\bOx)\]
	by Lemma~\ref{l:bOx-in-arbitrary-char}
	and we are thus reduced to proving the statement for $\O^\times$.
	
	Let now $L\in H^1_{\et}(X\times Y,\O^\times)$ and let $x\in X(K)$. Using the specialisation $Y\to X\times Y\to Y$ at $x$, we can after twisting by the pullback of $L_x^{-1}$ assume without loss of generality that $L$ specialises to the trivial bundle. In this case, the following lemma gives the desired result:
	\begin{Lemma}\label{l:spreading-out-triviality}
		Let $X$ and $Y$ be affinoid rigid spaces over $K$ and let $x\in X(K)$. Let $L$ be a line bundle on $X\times Y$ that is trivial on $\{x\}\times Y$. Then there is an affinoid open neighbourhood $x\in U\subseteq X$ such that $L$ is trivial on $U\times Y$.
	\end{Lemma}
	\begin{proof}
		Let $(U_i)_{i\in I}$ be the inverse system of open neighbourhoods of $x\in X$, then  we have
		\[ \{x\}\times Y\approx \varprojlim_{i \in I} U_i\times Y\]
		and consequently
		\[\O^+(\{x\}\times Y)=(\varinjlim_{i\in I} \O^+(U_i\times Y))^\wedge\]
		by \cite[ Lemma~3.10]{heuer-diamantine-Picard}.
		The statement thus follows from Proposition~\ref{p:Gabber-Ramero-5.4.42}.
	\end{proof}
	
	For part 1, we are left to deduce the result for general $Y$: Let $L\in H^1_{\et}(X\times Y,\bOx)$ and let $x\in X(K)$, then we can find a finite cover of $Y$ by affinoid opens $U_i$ such that $L$ is trivial on each $\{x\}\times U_i$. By rigid approximation and Lemma~\ref{l:spreading-out-triviality}, we can find an affinoid neighbourhood $x\in V\subseteq X$ such that $L$ is already trivial on each $V\times U_i$. We may without loss of generality replace $X$ by $V$. The \cH-to-sheaf sequence now reduces us to proving the lemma for $H^0$ instead of $H^1$, i.e.\ that
	\[H^0_{\et}(X\times Y,\bOx)\to \Map(X(K),H^0_{\et}(Y,\bOx))\]
	has image in the locally constant maps. Once again we can reduce to the case that $X$ and $Y$ are rigid. Here the statement follows from \cite[Lemma 4.10]{heuer-diamantine-Picard}: Applied to $Y$ and $X\times Y$, this asserts that specialisation induces  injective maps
	\[H^0_{\et}(Y,\bOx)\hookrightarrow  \Map_{\lc}(Y(K),K^\times/(1+\m))\] \[H^0_{\et}(X\times Y,\bOx)\hookrightarrow  \Map_{\lc}(X(K)\times Y(K),K^\times/(1+\m)),\]
	from which the  desired statement follows.

	To see part 2, let $\pi$ be the projection $X\times Y\to Y$, then $R^1\pi_{\et\ast}\bOx=0$ since we can compute this in the analytic topology by Lemma~\ref{l:bOx-in-arbitrary-char}.3. It follows that
	\[H^1_{\et}(X\times Y,\bOx)=H^1_{\et}(Y,\pi_{\ast}\bOx)=\textstyle\varinjlim_{\mathfrak U}\cH^1(\mathfrak U,\pi_{\ast}\bOx)\]
	where $\mathfrak U$ ranges through finite open covers of $Y$. We can now use that we have \[\bOx(X\times Y)=\Map_{\lc}(X(K),\bOx(Y)).\] Since the covers are finite and $\Map_{\lc}(X(K),-)$ is exact, it follows that 
	\[ \textstyle\varinjlim_{\mathfrak U}\cH^1(\mathfrak U,\pi_{\ast}\bOx) =\Map_{\lc}(X,\textstyle\varinjlim_{\mathfrak U}\cH^1(\mathfrak U,\bOx))=\Map_{\lc}(X(K),H^1_{\et}(Y,\bOx))\]
	which gives the desired isomorphism.
\end{proof}

\begin{Lemma}\label{l:factorisation-of-boundary-map}
	Let $T$ be an affinoid connected reduced rigid space. Then the boundary map of the exponential sequence has a factorisation 
	\[\begin{tikzcd}
		& {H^2_{\an}(X,\O)} \arrow[d, hook] \\
		{H^1_{\an}(X\times T,\bOx)} \arrow[r]\arrow[ru, dotted] & {H^2_{\an}(X\times T,\O)}.
	\end{tikzcd}\]
\end{Lemma}
\begin{proof}
	Using that $X$ is proper, it is easy to see (e.g.\ \cite[Proposition~4.2.1]{heuer-diamantine-Picard}) that we have 
	\[H^2_{\an}(X\times T,\O)=H^2_{\an}(X,\O)\otimes_K \O(T).\]
	Specialisation at points in $T(K)$ therefore defines a commutative diagram
	\[ \begin{tikzcd}
		{H^1_{\an}({X}\times {T},\bOx)} \arrow[d] \arrow[r] & {H^2_{\an}(X,\O)\otimes_K \O(T)} \arrow[d, hook] \\
		{\Map_{\lc}(T(K),H^1(X,\bOx))} \arrow[r] & {\Map_{\cts}(T(K),H^2(X,\O)),}
	\end{tikzcd}\]
	where we use Lemma~\ref{l:specialisation-for-bOx-in-products} to see that the image of the left map lies in the locally constant maps, and where the bottom map is defined as the composition with the boundary map $H^1(X,\bOx)\to H^2(X,\O)$ from the case of $T=\Spa(K)$. We conclude that the image of any element of the top left in the bottom right lands in the locally constant maps.
	
	But the map on the right is the evaluation $\O(T)\to \Map_{\cts}(T(K),K)$ tensored with the finite dimensional $K$-vector space $H^2(X,\O)$. Since $T$ is reduced, a function $f\in \O(T)$ can only be locally constant on $T(K)$ if $f$ is locally constant. As $T$ is connected, this implies $f\in K$. This shows that the image of the top map factors through the pullback $H^2(X,\O)\to H^2(X\times T,\O)$, as desired.
\end{proof}
\subsection{Proof of Theorem~\ref{t:Pic-functor-for-wtB}}
\begin{proof}[Proof of Theorem~\ref{t:Pic-functor-for-wtB}]
	We start by observing that in the proper case, the cohomological condition of Theorem~\ref{t:Pic-functor-for-wtB} automatically implies a relative version:
	\begin{Lemma}\label{l:rel-version-of-condition-(i)}
		For any affinoid perfectoid $Y$, we have 
		$H^j(X_\infty\times Y,\O^+)\aeq 0$ for $j=1,2$.
	\end{Lemma}
	\begin{proof}
		Using \cite[Lemma 3.11]{heuer-diamantine-Picard} and \Cref{p:cohom-comparison}, we have
		\[H^j_{\et}(X_\infty\times Y,\O^+/p^n)=\varinjlim_i H^j_{\et}(X_i\times Y,\O^+/p^n).\]
		By  \cite[Proposition~4.2.2(iii)]{heuer-diamantine-Picard} and the Primitive Comparison Theorem, this is 
		\[\aeq\varinjlim_i H^j_{\et}(X_i,\Z/p^n)\otimes_{\Z_p}\O^+(Y).\]
		Furthermore, using  \cite[Proposition~14.9]{etale-cohomology-of-diamonds} this is
		\[= H^j_{\et}(X_\infty,\Z/p^n)\otimes_{\Z_p}\O^+(Y)\]
		for all $j,n\geq 0$. Comparing to the case of $Y=\Spa(K)$, the assumptions imply
		\[  H^1_{\et}(X_\infty\times Y,\O^+/p^n)\aeq 0.\]
		For $j=1,2$ this shows $R^1\varprojlim_n H^{j-1}(X_\infty\times Y,\O^+/p^n)\aeq 0$. Thus by Corollary~\ref{c:replete-derived-complete},
		\[  H^j_{\et}(X_\infty\times Y,\O^+)=H^j_{v}(X_\infty\times Y,\O^+)\aeq \textstyle\varprojlim_n(H^j_{\et}(X_\infty,\Z/p^n)\otimes_{\Z_p}\O^+(Y)).\]
		Since $\O^+(Y)$ is flat over $\Z_p$, we can now choose an isomorphism $\O^+(Y)\cong\widehat\oplus_I \Z_p$ by lifting an $\F_p$-basis of $\O^+(Y)/p$. This shows that the right-hand-side almost vanishes if and only if \[\textstyle\varprojlim_nH^j_{\et}(X_\infty,\Z/p^n)=0.\]
		As the case of $Y=\Spa(K)$ holds by assumption, this gives the desired result.
	\end{proof}

	We now  follow the strategy from \cite[\S2.3 and \S3]{heuer-diamantine-Picard}:	Let $Y=\Spa(S,S^+)$ be affinoid perfectoid and consider the long exact sequence of the exponential sequence of Proposition~\ref{p:bOx} on $X_\infty\times Y$. By 
	\cite[Proposition~3.17]{heuer-diamantine-Picard}, we have $Y\approx \varprojlim_{Y\to T} T$ where  $Y\to T$ runs through (isomorphism classes of) all affinoid smooth rigid spaces over $K$ to which $Y$ maps.
	 Using first Lemma~\ref{l:rel-version-of-condition-(i)} and then again \cite[Lemma 3.11]{heuer-diamantine-Picard} and \Cref{p:cohom-comparison}, we see that
	\begin{equation}\label{eq:rel-thm-first-step}
	H^1_{\et}(X_\infty\times Y,\O^\times\tf)=H^1_{\et}(X_\infty\times Y,\bOx)=\varinjlim_{i}\varinjlim_{Y\to T} H^1_{\et}(X_i\times T,\bOx)
	\end{equation}

	Fix some $i\in I$ and consider $X:=X_i$. Let $T=\Spa(A,A^\circ)$ be any affinoid smooth rigid space over $K$ and set $\mathfrak T=\Spf(A^\circ)$. For any admissible formal scheme $\mathfrak S$ over $\O_K$ of topologically finite type, let us denote by $\mathfrak S_{\ZAR}$ the big Zariski site of admissible topologically finite type formal schemes over $\O_K$. For any rigid space $S$ over $K$, we denote by $S_{\AN}$ the big site of rigid spaces over $S$ with the analytic topology. We obtain a base-change diagram
	\[\begin{tikzcd}
		X_{\AN} \arrow[d] \arrow[r, "\pi_{\AN}"] & \Spa(K)_{\AN} \arrow[d, "\lambda"] \\
		\mathfrak X_{\ZAR} \arrow[r, "\pi_{\ZAR}"] & \Spf(\O_K)_{\ZAR}
	\end{tikzcd}\]
	where $\lambda$ sends $\mathfrak S$ to its generic fibre $\mathfrak S_{\eta}$. To avoid confusion in the notation, let us denote by $\O^+$ the structure sheaf on   $\Spf(\O_K)_{\ZAR}$ and let $\overbar{\O}^{+\times}:=(\O^{+\times}/(1+\mathfrak m\O^+))\tf$. This is simply the units of the structure sheaf of the special fibre over $\Spec(k)$, up to formally inverting $p$. We also set $\O:=\O^+\tf$ on  $\Spf(\O_K)_{\ZAR}$.
	\begin{Proposition}\label{p:base-change-morph-for-bOx}
		The following base-change morphisms are isomorphisms:
		\begin{enumerate}
			\item $\alpha:\displaystyle\lambda^{\ast}R^1\pi_{\ZAR\ast}\O^{+\times}\to R^1\pi_{\AN\ast}\O^\times$,
		\item $\beta:\displaystyle\lambda^{\ast}R^1\pi_{\ZAR\ast}\bOpx\to R^1\pi_{\AN\ast}\bOx$.
		\end{enumerate}
	\end{Proposition}
	\begin{proof}
		By definition, the map  $\alpha$ is the sheafification of the map that on $T\in (\Spa(K))_{\AN}$ evaluates to
		\begin{equation}\label{eq:map-in-5.10}
		\varinjlim_{\mathfrak T} H^1_{\Zar}(\mathfrak X\times \mathfrak {\mathfrak T},\O^{+,\times})\to H^1_{\an}(X\times T,\O^\times),
		\end{equation}
		where the colimit is over the cofiltered inverse system of formal models $\mathfrak T$ of $T$ with transition morphisms given by the admissible blow-ups. 
		\begin{claim}
			$\alpha$ is surjective.
		\end{claim}
		\begin{proof} 
			By \cite[Lemma~6.2.4]{Lutkebohmert_RigidCurves} applied to $\mathfrak X\times \mathfrak {\mathfrak T}\to \mathfrak T$, any line bundle on $X\times T$ admits a formal model on $\mathfrak X\times \mathfrak T$ after replacing $\mathfrak T$ by an admissible blow-up. Hence  \eqref{eq:map-in-5.10} is surjective, which implies that $\alpha$ is surjective.
		\end{proof}
		
		The first part is thus essentially the statement that the rigid generic fibre of the Picard functor of $\mathfrak X$ is the Picard functor of $X$: Namely, it is clear on $K$-points that the map
		\[ (\Pic_\mathfrak X)_{\eta}\to \Pic_{X}\]
		is injective. However, one can also give a more elementary argument that avoids the question whether $\Pic_\mathfrak X$ is representable by an admissible formal scheme. 
		
		For this we note that we can always cover $T$ by the generic fibre of an affine open covering of $\mathfrak T$. Therefore, the map \eqref{eq:map-in-5.10} has the same sheafification as the one defined on affinoids $T=\Spa(R,R^\circ)$ by
		\[H^1_{\Zar}(\mathfrak X\times \mathfrak T,\O^{+,\times})\to H^1_{\an}(X\times T,\O^\times),\]
		where $\mathfrak T:=\Spf(R^\circ)$.  For part (1), it remains to see that this sheafification is injective.
		
		To see this, we will first consider the sheaves $\O$ and $\bOx$ instead of $\O^\times$. Recall that we denote by $\O$ also the sheaf $\O^+\tf$ on $\Spf(\O_K)_{\ZAR}$.
		\begin{claim}\label{eq:bc-for-O}
			The following base-change morphism is an isomorphism for any $n\in \N$:
			\[
		\lambda^{\ast}R^n\pi_{\ZAR\ast}\O\isomarrow R^n\pi_{\AN\ast}\O\]
		\end{claim}
		\begin{proof}
			By the same argument as above, this is the sheafification of the map sending $T=\Spa(R,R^\circ)$ to 
		\[H^n_{\Zar}(\mathfrak X\times \Spf(R^\circ),\O^{+})\tf\to H^n_{\an}(X\times T,\O).\]
		This is an isomorphism as we can see by computing both sides in terms of \cH-cohomology for any affine cover of $\mathfrak X$.
		\end{proof}
		Next, we consider the sheaf $\bOx$ and prove:
		\begin{claim}
			 $\beta$ is injective.
		\end{claim}
		\begin{proof}
			  Set $\overbar{T}:=\Spec(R^\circ/\mathfrak m)$. By the same argument as above, $\beta$ is the sheafification of the morphism of presheaves given on $T=\Spa(R,R^\circ)$ by
		\[ H^1_{\Zar}(\mathfrak X\times \mathfrak T,\bOpx)\to H^1_{\an}(X\times T,\bOx).\]
		We now first note that by definition of $\bOpx$, the left hand side is $\Pic(\overbar{X}\times \overbar{T})\tf$. Upon specialisation at points in $T(K)$, the above map thus fits into a commutative diagram
		\[\begin{tikzcd}
			\Pic(\overbar{X}\times\overbar{T})\tf \arrow[r] \arrow[d] & {H^1_{\an}(X\times T,\bOx)} \arrow[d] \\
			{\Map(\overbar T(k),\Pic(\overbar{X})\tf)} \arrow[r, hook] & {\Map_{\lc}(T(K),H^1_{\an}(X,\bOx)).}
		\end{tikzcd}\]
		The left map has kernel $\Pic(\overbar{T})\tf$: Since $\overbar{X}\to \Spec(k)$ is smooth proper, its Picard functor is represented by a group scheme $\uP_{\overbar{X}}$. As by the Reduced Fibre Theorem of \cite{BLR-IV-Reduced-Fibre} the scheme $\overbar{T}$ is reduced, it follows that evaluation at the algebraically closed field $k$ induces an injection
		\[ 	\Pic(\overbar{X}\times\overbar{T})/\Pic(\overbar{T})=\Map(\overbar{T},\uP_{\overbar{X}})\hookrightarrow \Map(\overbar{T}(k),\Pic(\overbar X)).\]
		The bottom map is injective as $T(K)\twoheadrightarrow \overbar{T}(k)$ is surjective and $\Pic(\overbar{X})\tf\to H^1_{\an}(X,\bOx)$ is an isomorphism by Proposition~\ref{p:bOx-computes-line-bundles-on-special-fibre}.
		It follows that the top map is injective already after Zariski-sheafification on $\overbar{T}$. This proves that $\beta$ is injective.
		\end{proof}
		
		\begin{claim}\label{cl:alpha-inj}
			$\alpha$ is injective, hence an isomorphism.
		\end{claim}
		\begin{proof}
			We compare the exponential sequence 
\[ 0\to \O\xrightarrow{\exp} \O^\times\tf\to \bOx\to 0\]
on $\Spa(K)_{\AN}$ from Proposition~\ref{p:bOx} to its analogous integral version
\[0\to \O\xrightarrow{\exp} \O^{+\times}\tf\to \bOpx\to 0\]
 on $\Spf(\O_K)_{\ZAR}$. This results in a commutative diagram
 	\[\begin{tikzcd}[row sep ={0.4cm,between origins}, column sep ={1.2cm,between origins}]
 	R^1\pi_{\AN\ast}\O\arrow[rrr,hook,"\exp"]&&&{R^1\pi_{\AN\ast}\O^\times\tf} \arrow[rrr] &  &  & {R^1\pi_{\AN\ast}\bOx} \arrow[rrr] &  &  & {R^2\pi_{\AN\ast}\O}\\
 	&  &  &  &  &  &   \\
 	&&&&  &  &  &  & &  \\
 	\lambda^{\ast}R^1\pi_{\ZAR\ast}\O\arrow[rrr,hook]\arrow[uuu,"\sim"labelrotate]&&&{\lambda^{\ast}R^1\pi_{\ZAR\ast}\O^{+\times}\tf} \arrow[rrr] \arrow[uuu, "\alpha"'] &  &  & {\lambda^{\ast}R^1\pi_{\ZAR\ast}\bOpx} \arrow[rrr] \arrow[uuu,"\beta"',xshift=1.0ex,hook']&  &  & {\lambda^{\ast}R^2\pi_{\ZAR\ast}\O}. \arrow[uuu,"\sim"labelrotate]
 \end{tikzcd}\]
 We have already seen in \eqref{eq:bc-for-O} that the first and last vertical map are isomorphisms.
 The map $\exp$ in the top row is injective by \cite[Lemma~4.11]{heuer-diamantine-Picard}, thus so is the leftmost map on the bottom.
 The fact that $\beta$ is injective therefore implies that $\alpha$ is injective.
 \end{proof}
 \begin{claim}
 	$\beta$ is surjective, hence an isomorphism.
 \end{claim}
 \begin{proof}
	We first observe that for any sheaf $\mathcal F$ on $(\Spa(K))_{\AN}$, pullback along $X\times T\to X$ defines for any $T$ a natural map
	\[ H^1_{\an}(X,\mathcal F)\hookrightarrow H^1_{\an}(X\times T,\mathcal F)\]
	which is injective because it is split by any map $X\to X\times T$ induced by any $K$-points of $T$.
	Upon sheafification, this induces an injection, natural in $\mathcal F$,
	\[c_{\mathcal F}: \underline{H^1(X,\mathcal F)}\to R^1\pi_{\an\ast}\mathcal F\]
	where the left hand side denotes the associated locally constant sheaf on $(\Spa(K))_{\AN}$.
	
	The same works for $\mathfrak X\times \mathfrak T\to \mathfrak T$ and we  obtain an analogous morphism on $\Spf(\O_K)_{\ZAR}$.
	Combining this with the long exact sequences of $\exp$ in the proof of \Cref{cl:alpha-inj}, we arrive at a commutative diagram
	\[\begin{tikzcd}[row sep ={0.9cm,between origins}, column sep ={1.1cm,between origins}]
		&  & {\underline{H^1_{\AN}(X,\bOx)}} \arrow[rd,hook,start anchor={[yshift=0.3ex,xshift=0.8ex]south},end anchor={[xshift=3ex,yshift=-0.2ex]north west}] \arrow[rrr] &  &  & {\underline{H^2_{\AN}(X,\O)}} \arrow[rd, hook,start anchor={[yshift=0.3ex,xshift=1.5ex]south},end anchor={[xshift=3ex,yshift=0ex]north west}] \arrow[rrr] &  &  & {\underline{H^2_{\AN}(X,\O^\times\tf)}} \arrow[rd,hook,start anchor={[yshift=0.3ex,xshift=0.8ex]south},end anchor={[xshift=3ex,yshift=-0.2ex]north west},"c_{\O^\times[\frac{1}{p}]}"']\\
		{R^1\pi_{\AN\ast}\O^\times\tf} \arrow[rrr] &  &  & {R^1\pi_{\AN\ast}\bOx}\arrow[rru,dotted] \arrow[rrr] &  &  & {R^2\pi_{\AN\ast}\O}\arrow[rrr]&  &  & {R^2\pi_{\AN\ast}\O^\times\tf} \\
		&  & {\lambda^{\ast}\underline{H^1_{\ZAR}(\overbar{X},\bOpx)}} \arrow[rd, hook,start anchor={[yshift=0ex]south},end anchor={[xshift=5ex,yshift=-0.2ex]north west}] \arrow[uu, "\sim"{labelrotate,xshift=-6pt},xshift=-5pt] &  &  & &  \\
	{\lambda^{\ast}R^1\pi_{\ZAR\ast}\O^{+\times}\tf} \arrow[rrr] \arrow[uu, "\alpha"',"\sim"labelrotate] &  &  & {\lambda^{\ast}R^1\pi_{\ZAR\ast}\bOpx} \arrow[uu,"\beta"',xshift=2.2ex,hook']&  &  & {}
	\end{tikzcd}\]
	By Lemma~\ref{l:factorisation-of-boundary-map}, the boundary map on top factorises, giving the dotted arrow. Using that the rightmost diagonal map $c_{\O^\times[\frac{1}{p}]}$
	 is injective, it follows from a diagram chase that 
	 \[\underline{H^1_{\AN}(X,\bOx)} \oplus R^1\pi_{\AN\ast}\O^\times\tf\to R^1\pi_{\AN\ast}\bOx\]
	 is surjective.
	  From this it follows by another diagram chase that $\beta$ is an surjective.
	 \end{proof}
	 All in all, this shows that $\alpha$ and $\beta$ are isomorphisms.
\end{proof}

	We now put everything together to finish the proof of Theorem~\ref{t:Pic-functor-for-wtB}: In \eqref{eq:rel-thm-first-step}, we had seen
	\[ H^1_{\et}(X_\infty\times Y,\O^\times)\tf=\varinjlim_{i\in I}\varinjlim_T H^1_{\et}(X_i\times T,\bOx).\]
	By Lemma~\ref{l:bOx-in-arbitrary-char}, we have $H^1_{\et}(X_i\times T,\bOx)=H^1_{\an}(X_i\times T,\bOx)$. We compare this to formal models:
	Let $\mathfrak T=\Spf(\O(T)^\circ)$ as before. Then by Proposition~\ref{p:base-change-morph-for-bOx}.2, the map
	\[ H^1_{\Zar}(\mathfrak X_i\times \mathfrak T,\bOpx)\to H^1_{\an}(X_i\times T,\bOx)\]
	is an isomorphism after analytic sheafification in $T$, thus also after \'etale sheafification.
	Since $Y_{\etqcqs}=\twolim T_{\etqcqs}$, we deduce  that the following is an isomorphism after sheafification:
	\[ \varinjlim_{i\in I}\varinjlim_T H^1_{\et}(X_i\times T,\bOx) \leftarrow \varinjlim_{i}\varinjlim_TH^1_{\Zar}(\mathfrak X_i\times \mathfrak T,\bOpx).\]
	We recall that  by definition of $\bOpx$, we have $H^1_{\Zar}(\mathfrak X_i\times \mathfrak T,\bOpx)=\Pic(\overbar X_i\times \overbar{T})\tf$, which Zariski-sheafifies to $\uP_{\overbar X_i}(\overbar{T})\tf$.
	Second, we note that 	by \Cref{p:cohom-comparison}, we have
	\[ \O^+(Y)/\mathfrak m=\varinjlim_{Y\to T} \O^+(T)/\mathfrak m=\varinjlim_{Y\to T}\O(\overbar{T}),\]
	which implies that $\uP_{\overbar{X}_i}^\diamond(Y)=\varinjlim_T\uP_{\overbar{X}_i}(\overbar T).$
	All in all, this shows that
	\[ \varinjlim_{i}\varinjlim_TH^1_{\Zar}(\mathfrak X_i\times \mathfrak T,\bOpx)\to \varinjlim_{i}\varinjlim_T \uP_{\overbar{X}_i}(\overbar T)\tf\to   \varinjlim_{i}\uP_{\overbar{X}_i}^\diamond(Y)\tf \]
	is an isomorphism after sheafification in $Y_{\et}$.
	
	Putting everything together, this finishes the proof of the theorem.
\end{proof}

\subsection{Picard functors approximated by rigid Picard varieties}
So far, we have seen examples in which the map in \Cref{q:Pic(X)->Pic(X_infty)-surj} is far from an isomorphism. We now complement these results by giving a criterion for when the map does turn out to be an isomorphism:
\begin{Theorem}\label{t:Picard-bij}
	Let $K$ be a perfectoid field extension of $\Q_p$.
	Let $X_\infty\sim \varprojlim_{i\in I} X_i$ be a perfectoid tilde-limit of qcqs smooth rigid spaces. Assume that each $X_i$ admits a cover by affinoid opens $U_i$ such that $U_i\times_{X_i}X_j$ is affinoid for all $j\geq i$.  Suppose that the natural maps
	\[\textstyle\varinjlim_{i\in I} H^j_\an(X_i,\O)\to H^j_\an(X_\infty,\O)\]
	are isomorphism for $j=0,1$ and injective for $j=2$. Then also the map 
	\[\textstyle\varinjlim_{i\in I} \Pic(X_i)\isomarrow \Pic(X_\infty)\]
	is an isomorphism.
	If all $X_i$ are proper, then even the morphism of Picard functors on $\Perf_{K,\et}$ is an isomorphism
	\[\textstyle\varinjlim_{i\in I} \underline{\Pic}_{X_i}\isomarrow \underline{\Pic}_{X_\infty}.\]
\end{Theorem}
As a first application, observe that the Theorem applies in particular if all $X_i$ are affinoid and $X_\infty$ is affinoid perfectoid. Hence we recover
\Cref{r:previously-known}.1. We will see another interesting example in \Cref{p:ordinary-tower} below.
\begin{proof}
	We begin with a similar devissage as in \cite[\S2.3]{heuer-diamantine-Picard}: The long exact sequences of
	\[ 1\to \O^\times_1\to \O^\times\to \bOx\to 1\]
	induces a commutative diagram
	\[\begin{tikzcd}[column sep=0.2cm]
		H^0(X_i,\bOx)\arrow[r]&{H^1(X,\O_1^\times)} \arrow[r] & \Pic(X) \arrow[r] & {H^1(X,\bOx)} \arrow[r] & {H^2(X,\O_1^\times)} \\
		{\varinjlim H^0(X_i,\O_1^\times)}\arrow[r]\arrow[u]&{\varinjlim H^1(X_i,\O_1^\times)} \arrow[r] \arrow[u] & \varinjlim  \Pic(X_i) \arrow[u] \arrow[r] & {\varinjlim}H^1(X_i,\bOx) \arrow[u] \arrow[r] & {\varinjlim H^2(X_i,\O_1^\times)} \arrow[u]
	\end{tikzcd}\]
	where each $\varinjlim$ is over $i\in \N$.
	The first and fourth columns are isomorphisms by \Cref{p:cohom-comparison}. By the 5-Lemma, it thus suffices to prove that the second column is an isomorphism and the fifth column is injective.
	
	For this we use the short exact sequence of the logarithm
	\[ 0\to \mu_{p^\infty}\to\O_1^\times\xrightarrow{\log} \O\to 0\]
	which induces for any $n\in \N$ a commutative diagram
	\[\begin{tikzcd}[column sep =0.2cm]
		{H^{n-1}(X,\O)} \arrow[r]& {H^n(X,\mu_{p^\infty})} \arrow[r] & {H^n(X,\O_1^\times)} \arrow[r] & {H^n(X,\O)} \arrow[r] & {H^{n+1}(X,\mu_{p^\infty})} \\
		{\varinjlim\! H^{n-1}(X_i,\O)\!} \arrow[u] \arrow[r] &{\!\varinjlim\! H^n(X_i,\mu_{p^\infty}\!)\!} \arrow[r] \arrow[u] & {\!\varinjlim\! H^n(X_i,\O_1^\times)\!} \arrow[r] \arrow[u] & {\!\varinjlim\! H^n(X_i,\O)\!} \arrow[u] \arrow[r] & {\!\varinjlim\! H^{n+1}(X_i,\mu_{p^\infty}\!)\!} \arrow[u]
	\end{tikzcd}\]
	The second and fifth column are isomorphisms by \cite[Proposition 14.9]{etale-cohomology-of-diamonds}. For $n=1$, the first and fourth column are isomorphisms by assumption. Hence so is the middle map by the 5-Lemma. For $n=2$, the first column is an isomorphism and the fourth column is injective by assumption, hence the middle column is injective by the 4-Lemma.
	
	The statement  about Picard functors follows by exactly the same argument: For this we let $Y$ be any affinoid perfectoid space, then we replace $X$ by $X\times Y$ and $X_i$ by $X_i\times Y$ in the above discussion and use additionally \cite[Lemma~3.11, Proposition~4.2.2]{heuer-diamantine-Picard}.
\end{proof}

\section{Examples and counterexamples where...}
In this final section, we use our results to give some examples and counterexamples that answer related questions raised by other authors:

\subsection{perfection trivialises line bundles}

Theorem~\ref{t:Pic(wt B)} yields many examples of proper rigid spaces $X$ in characteristic $p$ for which
\[\Pic(X)\tf\to \Pic(X^{\perf})\]
is not injective. This gives a negative answer to the question posed in Das' thesis \cite[Question 7.0.2]{das2016vector}, motivated by $X=\P^1$, whether there is some kind of GAGA principle for perfections of proper smooth schemes over algebraically closed fields of characteristic $p$:

Indeed, if $X=B$ is an abelian variety with good supersingular reduction over any perfectoid field $K$, then $B^\perf=\wt B$ and we thus have a short exact sequence
\[ 0\to \wh{B}^\vee(K)\to \Pic(X)\to \Pic(X^{\perf}).\]
For example for $K=\C_p^\flat$, where $B(K)$ is topologically torsion by \cite[Proposition 2.23]{heuer-geometric-Simpson-Pic},  this means that only the coprime-to-$p$ torsion of $B^\vee(K)$ survives in $\Pic(X^\perf)$.

\subsection{not all line bundles on $X^\perf$ come from $X$}\label{s:more-LB-on-X^perf}
The same example shows that
$\Pic(X)\tf\to \Pic(X^{\perf})$
is not in general surjective: 

Let $E$ be an elliptic curve over $K$, then $X=E\times E$ has Picard rank $\rk \NS(X)=2+\rk\End(E)$.
Assume now that $K$ is algebraically closed of characteristic $p$ and let $E$ be a generically  ordinary elliptic curve over $K$ with supersingular reduction $\overbar{E}$. Then $\End(\overbar{E})$ is a quaternion algebra, thus $\rk \End(\overbar{E})=4$. On the other hand, $\End(E)=1$. Consequently,
\[\rk_{\Z[\frac{1}{p}]}\NS(X^{\perf})=3,\quad \rk_{\Z}\NS(X)=6.\]
Thus the cokernel of $\Pic(X)\tf\to \Pic(X^{\perf})$ is a $\Z\tf$-module of rank 3.
In general, for a smooth proper rigid space $X$, we expect this cokernel to be a finitely generated $\Z\tf$-module.

\subsection{$\Pic(X)$ is not $p$-divisible, neither $p$-torsionfree, neither discrete}
Let $X$ be a perfectoid space over $K$.
If $\Char K=p$, then $\Pic(X)$ is uniquely $p$-divisible since $\O^\times$ is. We shall now give an example that this can fail in characteristic $0$. 

This is interesting for two reasons: First, Bhatt--Scholze have shown that $\Pic(X)$ is uniquely $p$-divisible for affinoid perfectoid $X$ \cite[Corollary 9.7]{MR4502597}.
 Second, {Dorfsman-Hopkins}  has shown in \cite[Corollary~1.3]{DH2021untilting} that if $\Gamma(X,\O)$ is perfectoid, then $\Pic(X)$ is uniquely $p$-divisible if and only if the $\sharp$ map sets up a ``tilting correspondence of line bundles''
\[\sharp:\Pic(X^\flat)\isomarrow\Pic(X).\]
In particular, we thus give an example where this does not define a bijection.

To construct the counter-example,  let $q\in K^\times$, $|q|<1$ be such that we can find a compatible system of $p^n$-roots $q^{1/p^n}$ of $q$ in $K$ for all $n$. If $\Char K=0$, this means that $q$ is in the image of $\sharp:K^{\flat\times}\to K^\times$. For instance, we could take $K=\Q_p(p^{1/p^\infty})^\wedge$ and $q=p$.

Let $A=\G_m/q^{\Z}$ be the rigid Tate curve. Then we get a tower of ``canonical'' isogenies
\[\dots\xrightarrow{[p]}\G_m/q^{\frac{1}{p^n}\Z}\to\dots\xrightarrow{[p]}\G_m/q^{\frac{1}{p}\Z}\xrightarrow{[p]}\G_m/q^{\Z}.\]
If $\Char K=p$, then each transition morphism is the relative Frobenius morphism.

It is easy to see that \cite[Proposition 6.1]{perfectoid-covers-Arizona} holds in this case: There is a perfectoid space 
\[X:=A^{\perf}\sim \textstyle\varprojlim_{i}\G_m/q^{\frac{1}{p^i}\Z}.\]
 If $\Char K=p$, then $A^\perf$ is indeed the perfection. If instead  $\Char K=0$, one can still identify $A^\perf$ with the untilt of $A'^\perf$ where $A':= \G_{m,K^\flat}/q'^{\Z}$ for  $q'=(q^{1/p^n})\in K^\flat=\varprojlim_{x\mapsto x^p}K$.

Using the methods of this article, one can now show the following:
\begin{Proposition}\label{p:ordinary-tower}
	Write $X_i:=\G_m/q^{\frac{1}{p^i}\Z}$, then the natural map 
	\[ \textstyle\varinjlim_{i\in \N}\Pic(X_i)\to \Pic(X)\]
	is an isomorphism. In particular, $\Pic(X)$ fits into a short exact sequence
\[ 0\to q^{\Z[\frac{1}{p}]/\Z}\to K^\times/q^{\Z}\to \Pic(X)\to \Z\tf\to 0.\]
\end{Proposition}
This complements \Cref{t:Pic(wt B)} as it gives an example in terms of abelian varieties where the map in \Cref{q:Pic(X)->Pic(X_infty)-surj} is an isomorphism.

Like $K^\times$, the group $\Pic(X)$ is thus in general neither $p$-torsionfree nor $p$-divisible. Moreover, $\Pic(X)$ in this case has a natural structure of a non-discrete topological group.
\begin{proof}
	We first treat the case that $\Char(K)=0$. For this we wish to invoke \Cref{t:Picard-bij}. For this to apply we need to verify that for $n=0,1,2$, the following map is an isomorphism:
	\[ \textstyle\phi_n:\varinjlim_{i\in \N} H^n(X_i,\O)\to H^n(X,\O).\]
	To see this, we compute $H^n(X,\O)$ using the Cartan--Leray sequence for the cover $\wt A\to X$: This is Galois with group $\Gamma\cong \Z_p$. It therefore follows from Proposition~\ref{p:H^i(wtA,O^+)} that
	\[ H^n(X,\O)=H^n_\cts(\Gamma,K)=\begin{cases}K&\quad \text{ for }n=0,1,\\0& \quad \text{ for }n\geq 2.\end{cases}\]
	For $n=0$, this immediately implies that $\phi_0$ is an isomorphism.
	
	For $n=2$, we also have $H^n(X_i,\O)=0$ for all $i$ since $X_i$ is a curve. Hence $\phi_2$ is bijective.
	
	It remains to treat $\phi_1$: For this we have  $H^1(X,\O)=\Hom(\Gamma,K)=K$. On the other hand, $H^1_\an(X_i,\O)$ is a one-dimensional $K$-vector space and the transition maps are isomorphisms. It thus suffices to prove that $\phi_1$ is injective.
	To see this, observe that $\Gamma\subseteq T_pA$ is an anticanonical subgroup. Hence the statement follows from \cite[\S4.4, Corollary~4.20.1]{HMW-abeloid-Simpson}.
	
	This shows the first sentence of the Lemma.	
	 To deduce the second part, it remains to show that $\varinjlim_{i\in \N}\NS(X_i)=\Z\tf$. We may assume that $K$ is algebraically closed. Due to \Cref{l:colim[p]=colim p* on Pics}, it suffices to see that 
	\[ \varinjlim H^2_{\et}(X_i,\mu_{p})\]
	vanishes. But this equals $ H^2_{\et}(X,\mu_p)$, and by the same Cartan--Leray argument for $\wt A\to X$ as above we indeed have 
	\[H^2_{\et}(X,\mu_p)=H^2_\cts(\Gamma,\Z_p)=0.\]
	We thus obtain the desired exact sequence.
	
	The case of characteristic $p$ follows that of characteristic $0$: We have $A^{\perf\flat}=A'^\perf$. Since $H^0(X,\O_X)=K$ is perfectoid, \cite[Theorem~1.2]{DH2021untilting} shows that
	\[\textstyle \Pic(A'^\perf)=\varprojlim_{[p]} \Pic(A^\perf),\]
	from which the result follows using that $K^{\flat\times}=\varprojlim_{x\mapsto x^p} K^\times$. 
\end{proof}
We note that \Cref{t:Picard-bij} in fact shows more, namely it shows that in characteristic $0$, we have a short exact sequence of sheaves on $\Perf_K$ computing the Picard functor of $X$:
\[ 0\to \underline{q^{\Z[\frac{1}{p}]/\Z}}\to \G_m/q^{\Z}\to \underline{\Pic}_X\to \underline{\Z\tf}\to 0.\]
Thus $\underline{\Pic}_X$ is represented in this case by the quotient of a rigid group by a discrete group.
\subsection{the untilt $\sharp:\Pic(X^\flat)\to \Pic(X)$ is not a bijection}
In the previous example, one can now also make the  $\sharp$-map on line bundles explicit, and see that this has non-trivial kernel and cokernel, in line with Dorfsman-Hopkins' result: Namely, one sees by comparing descent data that the sharp map fits into an exact sequence
\[
\begin{tikzcd}
	K^\times/q^{\Z[\frac{1}{p}]} \arrow[r]                   & \Pic(X) \arrow[r]          & {\Z[\frac{1}{p}]} \arrow[r]           & 0 \\
	K^{\flat\times}/q'^{\Z[\frac{1}{p}]} \arrow[r] \arrow[u,"\sharp"] & \Pic(X^\flat) \arrow[r] \arrow[u,"\sharp"] & {\Z[\frac{1}{p}]} \arrow[u,equal] \arrow[r] & 0
\end{tikzcd} \]
where $X^\flat=A'^\perf$. This is discussed in more detail in
 \cite[Example 4.6]{DH2021untilting}.

In particular, there are in general line bundles on $A^\perf$ that do not lift to $\G_m^\perf$-torsors.

\end{document}